\documentclass[a4paper,12pt,oneside]{amsart}

\usepackage{amsmath,amsfonts,amssymb,amsthm,amsopn,amsxtra}

\newtheorem{Thm}{Theorem}[section]    
\newtheorem{Lem}[Thm]{Lemma}
\newtheorem{Pro}[Thm]{Proposition}
\newtheorem{Cor}[Thm]{Corollary}

\newtheorem*{Th}{Theorem 0 {\rm (\cite{Lo3a}\;Th.\,2)}}

\theoremstyle{definition}
\newtheorem{Def}[Thm]{Definition}
\newtheorem{varrem}[Thm]{}
\newtheorem{Ex}[Thm]{Example}
\newtheorem{Exs}[Thm]{Examples}

\theoremstyle{remark}
\newtheorem{Rem}[Thm]{Remark}
\newtheorem{Rems}[Thm]{Remarks}

\DeclareMathOperator{\ad}{ad}

\DeclareMathOperator{\Aut}{Aut}
\DeclareMathOperator{\nil}{nil}

\DeclareMathOperator{\GL}{GL}
\DeclareMathOperator{\rk}{rk}
\DeclareMathOperator{\Int}{Int}
\DeclareMathOperator{\Lie}{Lie}

\newcommand{\R}{\mathbb{R}}        
\newcommand{\N}{\mathbb{N}}        
\newcommand{\Z}{\mathbb{Z}}        
\newcommand{\C}{\mathbb{C}}        
\newcommand{\Q}{\mathbb{Q}}        
\newcommand{\FC}{{FC}^-}        %
\newcommand{\Cal}{\mathcal}
\newcommand{\fr}{\mathfrak}
\begin{document}
\title[]{On the Structure of Groups with Polynomial Growth IV}
\author{Viktor  Losert}
\address{Fakult\"at f\"ur Mathematik, Universit\"at Wien, Strudlhofg.\ 4,
  A 1090 Wien, Austria}
\email{Viktor.Losert@UNIVIE.AC.AT}
\date{24 March 2021}
\subjclass[2010]{Primary 22D05; Secondary 22E25, 22E30, 20F19, 20F24, 20F69,
20G20, 43A20}

\begin{abstract}
Recently (\cite{Lo3a}) we have shown a structure theorem for locally compact
groups of polynomial growth. We give now some applications on various growth
functions and relations to $\FC_G$\,-\,series.
In addition, we show some results on related classes of groups.
\end{abstract}
\maketitle

\baselineskip=1.3\normalbaselineskip
\setcounter{section}{-1}

\section{Introduction and Main results} 
\smallskip
Let $G$ be a locally compact (l.c.), compactly generated group.
$\lambda$ denotes a Haar measure on $G$ and  $V$ a compact neighbourhood
of the identity $e$, generating $G$. The group $G$ is said to be of {\it
polynomial growth}, if there exists $d\in \N$ such that
$\lambda (V^n)=O(n^d)$ \,for $n\in \N$\,. In \cite{Lo3a} we have given
a structure theorem which we restate here.

\begin{Th}  \label{th2}       
Let $G$ be a compactly generated l.c.\;group of polynomial growth having no
non-trivial compact normal subgroup. Then $G$ can be embedded as a closed
subgroup into a semidirect product \,$\widetilde G = \widetilde N\rtimes K$
such that $K$ is compact, $\widetilde N$ is a connected, simply connected
nilpotent Lie group, $K$ acts faithfully on $\widetilde N$\,,
\,$\widetilde G/G$ is compact and \,$\widetilde NG$ is dense in
$\widetilde G$\,.
\end{Th}
This generalizes a classical result of Gromov \cite{Gr} for the discrete case.
Furthermore (\cite{Lo3a}\;Th.\,3), a group $\widetilde G$ satisfying the
conditions of Theorem\;\ref{th2} is uniquely determined up to isomorphism.
We called it the {\it algebraic hull} of $G$ and
\,$\widetilde N$ the {\it connected nil-shadow} of $G$\,. For $G=N$ a
torsion free compactly generated nilpotent group, $\widetilde G=\widetilde N$
coincides with the Malcev completion $N_\R$ (see \cite{Lo3a}\;1.2 for
properties and references). In this paper we
will give some applications and also show some related results on
l.c.\;groups of polynomial growth.

In Section\;1 we treat some natural normal series of $\FC_G$\,-\,groups.
The descending central series for $\widetilde N$ gives rise to a special
$\FC_G$\,-\,series $(H_n)$ \,for $G$\,. Theorem\;\ref{th52} describes its
properties and gives a formula 
\,$d\,=\,\rk(G/N) + \sum\limits_{n \geq 1}\, n\;\rk(H_{n-1}/H_n)$ for the
{\it growth} of $G$ (i.e. the minimal growth exponent~$d$ above),
generalizing a well known formula for finitely generated nilpotent groups
(see also Remark\;\ref{rem53}\,(a)).
Then we treat an ascending series ({upper $\FC$-\,central series}),
starting with the $\FC$-\,centre $B(G)$\,. Theorem\;\ref{th15} shows,
that for a compactly generated l.c.\,group $G\,,\ B(G)$ is always closed.
This extends results of Tits (for Lie groups, using structure theory and
properties of linear groups) and M\"oller/Trofimov (for totally disconnected
groups, using graph theory). We give also a more algebraic argument for
Trofimov's result. Theorem\;\ref{th112} considers properties of the
$\FC$-\,hypercentre extending \cite{FF}\;Th.\,1 (for discrete groups).
In Theorem\;\ref{th65} we give an intrinsic characterization of the
groups $H_n$ by local growth properties (considering the powers of elements)
and there is also a characterization of the nil-radical by growth properties
(considering conjugacy classes). This is done by applying methods going back
to \cite{Gu} for connected nilpotent Lie groups.
Section\;3 treats a generalization of Theorem\;\ref{th2} to certain
non-compactly generated groups (Theorem\;\ref{th74} on bounded polynomial
growth, i.e. l.c.\,groups $G$ for which the growth of the compactly generated
open subgroups is uniformly bounded). This extends \cite{Wi}\;Th.\,3.1 for
discrete groups.
Section 4 gives applications to the characterization
of symmetry for weighted group algebras (Theorem\;\ref{th82}). The groups
$H_n$ are used to give a parametrical description for $G$\,,
generalizing Malcev's coordinates of the second kind. Finally, in
Section\;5, the ``limit spaces'' of Gromov (asymptotic cone) are described in
terms of the algebraic completion (Theorem\;\ref{th92}). In
Remark\;\ref{rem93}\,(c) we add some comments concerning Breuillard's
strong refinements on the asymptotic behavior of powers $B^n$ (\cite{Br}).
The condition of near polynomial growth is slightly weaker,
requiring $\lambda (V^{n_i})=O( n_i^d)$ just for a subsequence.
Theorem\;\ref{th95} shows that this already implies polynomial
growth, extending \cite{VW} (for discrete groups) and \cite{Tr2}
(for totally disconnected groups). It contains also a slight simplification
to the last part of the argument in \cite{Gr}\;p.\,69-71.
\smallskip
\section{Growth formula and $\FC_G$\,-\,series} 
\medskip
\begin{varrem}\label{di51}	       
Let $G$ be a l.c.\;group. If $\Cal B$ is a group acting on $G$ by
automorphisms, then $G$ is said to be an
{\it $\FC_{\Cal B}$\,-\,group} if the orbits
\,$\{\alpha (x)\! : \alpha \in \Cal B\}$ are relatively compact in $G$ for
all $x \in G$\,. For $\Cal B=\Int(G)$ the inner automorphisms,
$G$ is called an $\FC$-\,group.
In \cite{Lo2}\;Th.\,1, we have shown that a compactly generated l.c.\;group
$G$ has polynomial growth iff it has a finite {\it $\FC_G$\,-\,series},
i.e., there exists a series $G=G_0\supseteq G_1\supseteq
\dots \supseteq G_n=(e)$ of closed normal subgroups of $G$ such that
$G_i/G_{i+1}$ is an $\FC_G$\,-\,group for $i=0,\dots ,n-1$ (taking the
action of $G$ induced by the inner automorphisms). Thus
the compactly generated groups of polynomial growth coincide
with the class of compactly generated $\FC$-\,nilpotent groups.
  
For general groups $G,\Cal B$ as above, $B_{\Cal B}(G)$ \;(the
{\it $\FC_{\Cal B}$\,-\,centre of $G$}) denotes the union of
the relatively compact $\Cal B$-orbits in $G$\,. $B_{\Cal B}(G)$ is
always a subgroup. For $\Cal B=\Int(G)$ we write just $B(G)$ \,(the
$\FC$-\,centre, relatively compact conjugacy classes, sometimes denoted
as $G_{\FC}$).

Let $G$ be a compactly generated group of polynomial growth, $C$ its
maximal compact normal subgroup (\cite{Lo2}\;Prop.\,1). Then
Theorem\;\ref{th2} applies to \,$G/C$
giving the algebraic hull \,$\widetilde G = \widetilde N \rtimes K$\,. By
\cite{Gu}\;Th.\,I.4, $G$ and $\widetilde N$ \,(the connected nil-shadow of
$G/C$) have the same growth (in the sense of \cite{Gu}\;Def.\,I.1, see below).

For a quantitative version, consider the descending central series
(lower central series) \,$\bigl(C_n(\widetilde N)\bigr)$ of $\widetilde N$,
i.e., \,$C_0 (\widetilde N) = \widetilde N,\
C_{n+1} (\widetilde N) = [\widetilde N, C_n (\widetilde N)]$. Put\vspace{-1mm}
$$d\, =\, \sum_{n \geq 1}\; n\,
\dim\bigl(\,C_{n-1}(\widetilde N)/C_n (\widetilde N)\,\bigr)\,.$$
\vskip-2mm minus 2mm
Then by \cite{Gu}\;Th.\,II.1, $\widetilde N$ has growth of degree $d$,
hence this gives a formula for the growth of $G$. Thus, if $V$ is any compact
$e$-neighbourhood generating $G$, there exist \,$c_1, c_2 > 0$ \,such that
\ $c_1\,n^d \leq \lambda (V^n) \leq c_2\,n^d\quad \text{ holds for }\
n = 1, 2,\dots$ \;(see also \cite{Gu}\;L.\,I.5,\,I.6 to transfer the lower
estimates).

Note that in particular, this implies that the notions of polynomial growth
and strict polynomial growth (\cite{Gu}\;Def.\,I.2) coincide for arbitrary
compactly generated locally compact groups. A more direct argument for this
equivalence has been given in \cite{FG2}\;L.\,2.3. I found this result at
the end of 1992 and communicated it to several people. Among them, at a
conference in Istanbul 2002 to M. Leinert whom I also indicated the short
argument. He worked out the technical details and published it without further
notice \dots\,. Meanwhile much more is known.\vspace{.3mm}
Breuillard (\cite{Br}), extending earlier results of Pansu, has
shown that \;$\lim\limits_{n\to\infty}\dfrac{\lambda (V^n)}{n^d}$ \;exists for
every $G,V$ as above (see also Remark\;\ref{rem93}\,(c)\,).
\vskip 2mm plus .5mm
We want to give a corresponding formula for the growth in terms of
subgroups of~$G$. We introduce a special $\FC_G$\,-\,series ($H_n$).

If $A$ is a compactly generated abelian group, we have
$A \cong \R^a \times \Z^b \times C$, where $C$ is compact, $a,b \geq 0$ \
(\cite{HR}\;Th.\,9.8). We
set \,$\rk(A) = a + b$\,, the (topological) {\it rank} of $A$\,. By
\cite{HR}\;Cor.\,9.13,
this is well defined. If $H$ is a compactly generated $\FC$-\,group,
$C$ its maximal compact normal subgroup, then by \cite{GM}\;Th.\,3.20,
$H/C$ is abelian
and we put \,$\rk(H) =\rk(H/C)$. More generally, if $H$ is any
compactly generated group and $H/B(H)$ is
compact, we put \,$\rk(H) = \rk\bigl(B(H)\bigr)$ \
(see also Remark\;\ref{rem53}\,(a) for a more general definition and
Theorem\;\ref{th15}).
\end{varrem}
\begin{Thm}  \label{th52}      
Let $G$ be a compactly generated group of polynomial growth having no
non-trivial compact normal subgroups. Let $N$ be the (non-connected)
nil-radical, 
$\widetilde N$ is given by Theorem\;\ref{th2}. Put
\,$H_n = G \cap C_n (\widetilde N)\quad (n = 0, 1, 2, ...)$.
Then the following holds.
\\[.6mm]
{\rm (i)} \
$d\,=\,\rk(G/N) + \sum\limits_{n \geq 1}\, n\;\rk(H_{n-1}/H_n)$
\ gives the growth of $G$\,.
\\[-.8mm]
{\rm (ii)} \ $H_0 = N$; \ for  $n \geq 1$ \,we have:
\;$C_n (\widetilde N)/H_n$
is compact, $H_n$ is normal in $G\,,\linebreak
 H_{n-1}/H_n$ is a torsion free 
$\FC_G$\,-\,group and contained in the centre of $N/H_n$\,,\linebreak
$H_n/({[G,G]}^-\cap H_n)$ is compact,
$G/H_n$ has no non-trivial compact normal subgroups.\vspace{-1mm}
\end{Thm}\noindent
$N=\nil(G)$ is the maximal nilpotent normal subgroup of $G$
(\cite{Lo2}\;Prop.\,3).
In the general case, we can pass to \,$G/C$ (see the beginning of \ref{di51})
and get a normal series in $G$ with similar properties and a corresponding
expression for the growth of $G$ (see Remark\;\ref{rem53}\,(a)\,).
\begin{proof}
We start with (ii). By
\cite{Lo3a}\;Prop.\,4.8\,(c), we have $N = G\cap\widetilde N = H_0$\,.
Then by \cite{Lo3a}\;Cor.\,4.6,
$H_1 = N \cap [\widetilde N,\widetilde N]$ is
co-compact in $C_1(\widetilde N) = [\widetilde N, \widetilde N]$.
Normality of $H_n$ in $G$ results from normality of $\widetilde N$ in
$\widetilde G$.
\\
We take up now the notations of \cite{Lo3a}\;Sec.\,4 and use some
additional subgroups of $\widetilde G$ considered there.
Recall that $\widetilde N/N$ is not compact in general
(\cite{Lo3a}\;Cor. 4.10\,(c)), in \cite{Lo3a}\;Prop.\,3.8 we have shown
existence of a l.c.\,group $G_{an}$ containing $G$ as a closed subgroup such
that $G_{an}/G\,,\ G_{an}/N_{an}$ are both compact, where
$N_{an}=\nil(G_{an})$ \,(i.e. $G_{an}$ is almost nilpotent),
$[N_{an}, N_{an}] \subseteq N$\,.
$G_{an}$ has no non-trivial compact normal subgroup, hence by
\cite{Lo3a}\;Th.\,3, the algebraic hull $\widetilde G$ of
$G_{an}$ is also an algebraic hull for
$G$ and by \cite{Lo3a}\;Cor.\,4.10\,(c), it follows that
$\widetilde N$/$N_{an}$ is compact \,(in general, $G_{an}$ is not unique,
given $\widetilde G = \widetilde N\rtimes K$\,, the ``standard choice'' is
$G_{an} = GK_1$ where $K_1=Z(K^0)^0$ ($H^0$ denoting the identity component,
$Z(H)$ the centre of $H$)\,,
see \cite{Lo3a}\;Rem.\,4.11\,(b) and \cite{Lo3a}\;Prop.\,4.4\,(a).
\\[1mm plus .8mm]
By \cite{Ra}\;Th.\,2.3\;Cor.1, $C_n(N_{an})$ is co-compact in
$C_n (\widetilde N)$. Since $[N_{an}, N_{an}] \subseteq N$, we get
$C_n (N_{an}) \subseteq  H_n$
for $n \geq 1$ and it follows that $H_n$ is co-compact in
$C_n (\widetilde N)$. Consequently, $G\,C_n(\widetilde N)$ is closed
and (by \cite{HR}\;Th.\,5.33) the canonical 
projection $\widetilde G \to \widetilde G/C_n(\widetilde N)$
induces a continuous isomorphism of $G/H_n$ onto a closed subgroup of
$\widetilde G/C_n(\widetilde N)$.
This maps $N/H_n$ into $\widetilde N/C_n(\widetilde N)$ and
$H_{n-1}/H_n$ to a closed subgroup of
$C_{n-1} (\widetilde N)/C_n (\widetilde N)$. $\widetilde N$
being a connected, simply connected Lie group, it follows that
$\widetilde N/C_n(\widetilde N)$ is
torsion free (\cite{Va}\;Th.\,3.18.2),
$C_{n-1}\bigl(\widetilde N/C_n (\widetilde N)\bigr)$
is a central subgroup, hence a $\FC_{\widetilde G}$\,-\,group. This gives the
corresponding properties of $H_{n-1}/H_n$ as claimed in (ii). By an
easy induction, semisimplicity implies that the action of $K$ induced on
$\widetilde N/C_n(\widetilde N)$ is faithful. Thus by
\cite{Lo3a}\;Prop.\,4.8\,(c), $G/H_n$ has no non-trivial compact normal
subgroups.
\\[0mm plus 1mm]
Concerning (i), recall that by \cite{Gu}\;Th.\,I.4, $G$ and $\widetilde G$
have the same growth. Thus we will relate the parts of the sum given in (i)
to those of the sum given in \ref{di51}. Using (ii),
\,$H_{n-1}/H_n$ is for $n \geq 2$ isomorphic to a closed co-compact
subgroup of
\,$C_{n-1} (\widetilde N)/C_n (\widetilde N)$, \,hence by
\cite{Ra}\;Th.\,2.10,
\,$\rk (H_{n-1}/H_n) = \dim\bigl(C_{n-1}(\widetilde N)/C_n(\widetilde N)\bigr)
\linebreak = \rk\bigl(C_{n-1}(\widetilde N)/C_n(\widetilde N)\bigr)$
holds for $n \geq 2$.
For the remaining terms, we can (after factoring
$[\widetilde N, \widetilde N]$\,)
assume that $\widetilde N$ is abelian. For $K_1$ as above, we consider
$\widetilde M = \{x\in\widetilde N\! :
k \circ x = x \text{ \,for all } k \in K_1\}\,,\
G_1= G \cap (\widetilde N \rtimes K_1)\,,\
L_1= G \cap (\widetilde M \rtimes K_1)$. Then by  \cite{Lo3a}\;Cor.\,4.9,
$L_1$ is abelian, $G_1 = NL_1$ is co-compact in $G$. Thus
$G_1/N$ is an abelian co-compact subgroup of $G/N$, hence it is contained
(and co-compact) in
$B (G/N)$. In particular, $G/N$ is always covered by the definition of
\,$\rk$ \,in \ref{di51}
and we have $\rk (G/N) = \rk (G_1/N) = \rk\bigl(L_1/(L_1 \cap N)\bigr)$. By
\cite{Lo3a}\;Prop.\,4.4\,(a), $L_1 K_1 \cap \widetilde M$ is
co-compact
in $\widetilde M$. We have $L_1 \subseteq  \widetilde M \times K_1$\,.
Projection to the first coordinate induces an isomorphism between
$L_1/(L_1 \cap NK_1)$ and
$(L_1 K_1 \cap \widetilde M)/(N \cap \widetilde M)$. Since
$N \cap \widetilde M \subseteq L_1$, we have
$L_1\cap\mspace{.5mu}NK_1 = L_1 \cap (N \cap \widetilde M) K_1 =
(L_1 \cap N) (L_1 \cap K_1)$, hence
$(L_1 \cap NK_1)/(L_1 \cap N)$ is compact.
By elementary properties of the rank of abelian groups
(compare \cite{Ra}\;Prop.\,2.8), it follows that 
\;$\rk\bigl(L_1/(L_1 \cap N)\bigr) =\linebreak
\rk\bigl((L_1 K_1 \cap \widetilde M)/(N \cap \widetilde M)\bigr) =
\rk\bigl(\widetilde M/(N \cap \widetilde M)\bigr)$\,. This implies
\;$\dim(\widetilde N) =\linebreak
\rk\bigl(\,\widetilde M\,/(N \cap \widetilde M)\bigr) + \rk (N) =
\rk (G/N) + \rk (N)$ \,which gives the remaining part of the sum in (i).
\vspace{.5mm plus 2mm}
\end{proof}
\begin{Rems} \label{rem53}	    
(a) \ The notion of the (topological) {\it rank} can be extended to the class
of generalized $\overline{FC}$-groups $G$
(including the compactly generated groups of polynomial growth,
see \cite{Lo2}\;1.2.1 for definition) as follows. If $G$ is
solvable and has no non-trivial compact normal subgroup,
put \;$\rk (G) = \sum_{n \geq 0}\,\rk\bigl(D_n (G)/ D_{n+1}(G)\bigr)$,
\,making use of the series of topological commutator groups
$(\,D_0(G) = G,\ D_{n+1}(G) = {[D_n(G), D_n(G)]}^-)$.
For general $G$ having no non-trivial compact normal subgroup, put
\,$\rk (G) = \rk (R)$, where $R$ denotes the (non-connected) radical of $G$
(\cite{Lo2}\;Prop.\,3). Finally, if $G$ is any generalized
$\overline{FC}$-group, put \,$\rk (G) = \rk (G/C)$,
where $C$ denotes the maximal compact normal subgroup of $G$
(\cite{Lo2}\;Prop.\,1).\vspace{.5mm plus 1mm}

Then by arguments as in the proof of Theorem\;\ref{th52}, one can show that if
$G$ has polynomial growth and no non-trivial compact normal subgroups,
$\widetilde G$ denotes its algebraic hull, then
\,$\rk(G) = \rk (\widetilde G) = \rk (\widetilde N) = \rk (G_{an}) =
\rk (N_{an})$. For $\widetilde N$ a connected, simply connected nilpotent Lie
group, the formula for the growth can be written as \,
$d=\sum_{n \geq 0}\,\dim\,\bigl(C_n (\widetilde N)\bigr)$.
If $N$ is any compactly generated nilpotent
group, one gets \ $d=\sum_{n \geq 0}\,\rk\,\bigl(C_n (N)\bigr) $.
For $G$ a compactly generated group of polynomial growth, Theorem\;\ref{th52}
gives \,$d=\rk(G)+\sum\limits_{n\geq 1}\,\rk\,(\,H_n) $.\vspace{0mm plus .2mm}

For finitely generated nilpotent groups, the rank is defined in
\cite{Ra}\;Def.\,2.9 and for (discrete) polycyclic groups in
\cite{Ra}\;Def.\,4.3. This
is also called ``Hirsch length'' (\cite{Se}\;p\,.16). The case $G$ connected,
simply connected, solvable is done in \cite{Ra}\;4.35. \cite{Mo}\;p.\,15
gives a definition for ``elementary solvable Lie groups'', see also\linebreak
\cite{Gu}\;p.\,347. Be aware that \cite{Man}\;p.\,83 and \cite{Ro} use a
different notion of rank.\vspace{1mm plus.2mm}

For nilpotent groups, the growth-formula (sometimes called formula of Bass)
has been obtained by various people (see \cite{Ha}\;p.\,201 for further
references).
\vspace{.5mm plus .9mm}
\item[(b)] For $n\geq 2$, $H_n$ can be characterized as the minimal
subgroup of $H_{n-1}$ such that $H_n$ is a closed normal subgroup of $G$
and $H_{n-1}/H_n$ is a torsion free $\FC_G$\,-\,group\linebreak
(if $D$ is any closed $G$-invariant subgroup of $H_{n-1}$ consider the
Malcev-completion $D_{\R} \subseteq \widetilde N$ \,(\cite{Lo3a}\;1.2);
then $H_{n-1}/D$ is an $\FC_G$\,-\,group
iff $C_{n-1}(N)/D_{\R}$ is an $\FC_{\widetilde G}$\,-\,group; furthermore,
$H_{n-1}/D$ torsion free implies $D_{\R} \cap H_{n-1} = D$\,).
\vspace{.9mm plus.2mm}

In a constructive way, $H_n$ can thus be obtained as follows. Let $D_0$
be the minimal closed subgroup of $H_{n-1}$ such that 
$D_0 \supseteq [H_{n-1}, H_{n-1}]$ and $H_{n-1}/D_0$ is torsion
free \,(in other words, $D_0/{[H_{n-1}, H_{n-1}]}^-$ is the maximal
compact subgroup of 
$H_n/{[H_{n-1}, H_{n-1}]}^-$, equivalently,
$D_0 = ({[H_{n-1}, H_{n-1}]}^-)_{\R} \cap H_{n-1}$\,). Then
$A = H_{n-1}/D_0$ is abelian and we write it 
additively. For $x \in G$, the automorphism of $A$ induced by the inner
automorphism $\iota_x$ is
again denoted by $\iota_x$. Then $H_n/D_0$ is the closed subgroup 
generated by all \,$(\iota_x)_u (v) - v$\,, where $x \in G\,, v \in A$ and
$(\iota_x)_u$ denotes the unipotent component of the Jordan
decomposition of $\iota_x$\,, as in \cite{Lo3a}\;2.5.
\vspace{.4mm plus .8mm}
\item[(c)] We say that $G$ has {\it abelian} growth, if its growth equals the
rank \,$\rk (G)$. By Theorem\;\ref{th52}\,(i), this is equivalent to
$\widetilde N$ being abelian and by \ref{th52}\,(ii), this is equivalent to
$H_1$ being trivial.
An intrinsic characterization of these groups has been given in
\cite{Lo3a}\;Cor.\,4.10\,(d). As shown
in \cite{Lo3a}\;Ex.\,4.12\,(e), it is not enough to assume that $N$ is an
$\FC_G$\,-\,group (by \cite{Lo3a}\;Cor.\,4.10\,(d) the condition is
necessary).
\ This implies that (differently from the corresponding statements in (b)\,)
in general $H_1$ is {\it not} the minimal subgroup of $N$ such that $H_1$
is a closed subgroup of $G$ and $N/H_1$ is an $\FC_G$\,-\,group
(see Example\;\ref{ex54}\,(b)\,). Of course, it is not hard to see that
$H_1$ is the
minimal subgroup of $N$ such that $H_1$ is a closed normal subgroup of $G$
and $G/H_1$ has abelian growth.\vspace{.3mm}

To have a construction for $H_1$, start as in (b), let $D$ be the minimal
closed $G$-invariant subgroup of $N$ such that $N/D$ is torsion free,
abelian and $(\iota_x)_u$ is
the identity on $(N/D)_{\R}$ for all $x\in G$. Consider $L_1$ as in
\cite{Lo3a}\;Cor.\,4.9. Then it is not hard to see that
$H_1 = \bigl({(D\,[L_1, L_1])}^-\bigr)_{\R} \cap N$ \;(see the proof of
\cite{Lo3a}\;Cor.\,4.6).
\vspace{.4mm plus 1mm}
\item[(d)] $(H_n)$ defines an $\FC_G$\,-\,series for $N$ satisfying
the properties of \cite{Lo2}\;Th.\,2 (but the step from $G$ to $N$ has to
be considered separately).
From the statements in (c), it follows that $(H_n)$ is in general not
the minimal $\FC_G$\,-\,series with these properties (see also
Example\;\ref{ex54}\,(b)).
By the argument in the proof of Theorem\;\ref{th52}\,(i),\linebreak
we have for $n\ge1$: \ $\widetilde N/C_n(\widetilde N)\rtimes K$ is
(isomorphic to) the algebraic hull of $G/H_n$\,, this gives
(by \cite{Lo3a}\;Prop.\,4.8\,(c)\,) \,$N/H_n=\nil(G/H_n)$ is
the non-connected nil-radical,
furthermore
\,$C_{n-1}(\widetilde N)/C_n(\widetilde N)\cong (H_{n-1}/H_n)_{\R}$\,.
\vspace{.5mm}

To compute the growth of $G$, the group $H_n$ can be replaced by
any closed subgroup $H'_n$ such that $H_n/H'_n$ is compact and
$H'_{n+1}$ is normal in $H'_n \ (n = 0,1,2,...)$. For example,
one can take $H'_n = C_n(N_{an})$. In general,
$C_n (N_{an}) \not= H_n$, thus
$C_n (N_{an})/C_{n+1}(N_{an})$ need not be torsion free (see
Example\;\ref{ex54}\,(a)). If $G$ is discrete (hence $N_{an} = N$), then
$H_n$ is the {\it isolator} of 
$C_n(N)$ in $N$ (\cite{Ba}\;p.\,19, \cite{War}\;Def.\,3.26). For general
$G$, $H_n$ is the "topological isolator" of $C_n (N_{an})$ in $N$, i.e.,
$H_n/C_n (N_{an})$ is the maximal compact normal subgroup of
$N/C_n (N_{an})$ for $n \geq 1$.\vspace{.5mm}

Alternatively, the growth of $G$ can be computed from the connected nil-shadow
$\widetilde N$, in the connected case, this goes back to
\cite{Gu}\;Th.\,II.2'.\vspace{-.4mm plus.5mm}
\end{Rems}
\begin{Exs}	\label{ex54}	    
{\bf (a)} \,Consider the groups \,$G = \R^n \rtimes \Z$ \,or
\,$\R^n \rtimes \R$
of \cite{Lo3a}\;Ex. 4.12\,(a). $n \circ v = A^n v$\,, for $v\in\R^n$,
where $A \in \GL(n,\R)$ and all
eigenvalues of $A$ have modulus $1$. $A$ is semisimple
(i.e., $A_u = I$), then $G$
has abelian growth, the growth equals $\rk(G) = n + 1$\,, but unless the
eigenvalues are roots of unity, $G$ does not have a co-compact abelian
subgroup.
\\[.6mm]
For \,$G = \C^2 \rtimes \Z$, with the action
$n \circ (z, w) = \alpha^n(z, w+n z)$ \,(which is neither semisimple nor
unipotent), we get \,
$N = \C^2\,,\;H_1 = (0) \times \C$\,. This $G$ has non-abelian growth,
$\rk (G) = 5$ and the growth formula gives $1 + 2 + 2 * 2 = 7$.
\\[.6mm plus .6mm]
Finally, consider 
\,$G = (\Z \times \R) \rtimes \Z$ with the action
\,$n \circ (k, x) = (k, x + nk)$ \;(another variation of the Heisenberg group).
Then \,$G = N = N_{an}$ is nilpotent,
$H_1 = (0) \times \R \times (0)$, but \,$C_1 (G) = (0) \times \Z \times (0)$
and the growth of $G$ is 4.
\vspace{1mm plus .9mm}
\item[\bf (b)] In \cite{Lo3a}\;Ex.\,4.12\,(e): $G = \mathsf H_{\Z}\ltimes \C$
\,($\mathsf H_{\Z}$\,: the discrete Heisenberg group, set theoretically
identified with $\Z^3$), with multiplication
\;$(k, l, m, z)\:(k', l', m', z') = \linebreak
(k + k', l + l', m + m' + l\,k',e^{i(k'\alpha + l'\beta)}z + z')$
\,where $\alpha,\beta\in\R$ with $\alpha,\beta,2\pi\
\,\Q$-linearly independent
(here \cite{Lo3a}\;p.\,38 is slightly incorrect).
We get
\,$H_1 = Z(\mathsf{H}_{\Z}) \cong \Z$ and the growth
of $G$ equals~6, whereas \,$\rk (G) = 5$ and the action of $G$ on
$N=\{(0,0,m,z)\! : m \in \Z,\, z \in \C\}$ has
relatively compact orbits, i.e., $N = B(G)$ is an $\FC_G$\,-\,group, the
action of $\mathsf H_{\Z}$ on $\C$ is semisimple. 
\end{Exs}

We include a result on the $\FC$-\,centre and also on the $\FC$-\,hypercentre.
I~thank the referee for providing reference \cite{FF} which is also used in
recent work on the Choquet-Deny property.
\begin{Thm}   \label{th15}	    
Let $G$ be a compactly generated l.c.\,group,
then $B(G)$ is a closed subgroup.
\end{Thm}\noindent
In \cite{Ti} this was shown when $G$ is a projective limit of Lie groups
(not necessarily connected or compactly generated; in fact his proof works
whenever $G$ has a compact normal subgroup $K$ such that $G/K$ is a Lie
group), in \cite{Mo1} for totally
disconnected (and compactly generated) $G$\,.
For the proof we use two auxiliary results.\vspace{-2mm}
\begin{Pro}	\label{pr16}	
Let $G$ be a Lie group, $\Cal B$ a subgroup of $\Aut(G)$ containing $\Int(G)$.
Then $B_{\Cal B}(G)$ is a closed subgroup of $G$\,.
\end{Pro}
\begin{proof}
We consider the semidirect product $G_1=G\rtimes\Cal B$\,, with discrete
topology on $\Cal B$\,. Then $G_1$ is again a Lie group, hence by
\cite{Ti}\;Cor.\,1,  $B(G_1)$ is closed.
It is easy to check that $B_{\Cal B}(G)=B(G_1)\cap G$\,.
\end{proof}
\begin{Lem}	 \label{le17}	
Let $H$ be an almost connected l.c.\,group.
Then there exists a maximal compact normal subgroup $K$\,. It is unique and
characteristic in~$H$\,, $KH^0$ is open in $H$\,.
$K\cap H^0$ is the maximal compact normal subgroup of~$H^0$,
$H^0/(K\cap H^0)$ is a Lie group with no non-trivial compact normal
subgroups.\vspace{-1mm}
\end{Lem}\noindent
$H^0$ denotes the identity component of $H$\,.
\begin{proof}
This is well known, see also \cite{WY}\;L.\,4. If $H$ is an almost connected
Lie group, existence of $K$ follows by first maximizing the dimension of
$K$ and then minimizing the degree $[H\!:\!KH^0]$ \,(one can also use more
advanced results of Iwasawa). Then the general case can be done using Yamabe's
theorem. The properties are easy.
\end{proof}\vspace{.5mm}
\begin{Pro}	\label{pr18}	
Let $G$ be a l.c.\,group such that $B(G/G^0)$ is closed in $G/G^0$. Then
$B(G)$ is closed in $G$\,, \,$G^0B(G)$ is open in $\pi_{G^0}^{-1}(B(G/G^0))$.
\vspace{-1.5mm}
\end{Pro}\noindent
$\pi_{G^0}\!:G\to G/G^0$ denotes the canonical projection. It is quite
obvious that $B(G)\subseteq\pi_M^{-1}(B(G/M))$ for any closed normal subgroup
$M$ of $G$\,.
\begin{proof}
Let $M$ be any l.c.\,group, $\Cal B$ some subgroup
of $\Aut(M)$ containing $\Int(M)$\,. If $K_0$ is a relatively compact
$\Cal B$-invariant subset of $M$ \,and periodic (i.e. each $x\in K_0$ belongs
to some compact subgroup of $M$), then the closed subgroup generated by
$K_0$ is compact and  $\mathcal{B}$-invariant (\cite{GM}\;Th.\,3.11(1)\,).
Furthermore, if $M$ is
an $\FC_{\Cal B}$\,-\,group and $K_1$ is any relatively compact subset of $M$
then by \cite{Wa}\;Prop.\,1.3, $K_0=\bigcup\limits_{\beta\in\Cal B}\beta(K_1)$
is relatively compact. If $K_1$ is also periodic it follows that the closed
$\Cal B$-invariant subgroup generated by $K_0$ is compact.
\\
Let $K$ be the maximal compact normal subgroup of $G^0$.
This is also normal in $G$ and by \cite{Ti}\;Prop.\,1 it will be enough
to prove the Theorem for $G/K$, i.e., we may assume that $K$ is trivial.
Then $G^0$ is a Lie group.
\\
By assumption $M=B(G/G^0)$ is a
closed subgroup of $G/G^0$ and clearly normal. 
Let $\Cal B$ be the subgroup of $\Aut(M)$ obtained by restricting the
inner automorphisms of $G/G^0$ to $M$. Then $M$ is an
$\FC_{\Cal B}$\,-\,group.
$M$ is totally disconnected, hence it has an open compact subgroup
$K_1$\,.  Let $K_2$ be the closed subgroup generated by
$\bigcup\beta(K_1)$. By the observation above, $K_2$ is compact,
it is clearly open in $M$ and normal in $G/G^0$ \,(this might have been also
obtained from \cite{WY}\;Th.\,1, similarly as in the proof of
Proposition\;\ref{pr16}) .
\\
We have $K_2=H/G^0$ for some closed normal subgroup $H$ of $G$. Then
$H$ is open in $\pi_{G^0}^{-1}(B(G/G^0))$ and
$H/G^0$ is compact. Let $K$ be the maximal compact normal
subgroup of $H$ (Lemma\;\ref{le17}). Then $G^0K$ is open in $H$ (clearly, we
have $H^0=G^0$).
$K$ is normal in $G$ and this implies $K\subseteq B(G)$\,.
\\
Applying Proposition\;\ref{pr16}, by taking $\Cal B_1$ all restrictions of
inner automorphisms of $G$ to $G^0$\,, we see that
\,$B(G)\cap G^0=B_{\Cal B_1}(G^0)$ is closed. This implies that\linebreak
$B(G)\cap (G^0K)=(B(G)\cap G^0)\,K$ is closed. $G^0K$ being open in
$\pi_{G^0}^{-1}(B(G/G^0))$, it follows that $B(G)$ is closed
\,(e.g. by \cite{HR}\;Th.\,5.9).
\end{proof}
We want to translate now Trofimov's graph theoretical setting (in particular
\cite{Mo2}\;L.\,5) into algebraic combinatorial language, continuing
work done in \cite{Mo2}, \cite{Mo1}.
\begin{Lem}	 \label{le19}	
Let $H$ be a compactly generated, totally disconnected l.c.\,group,
$K$~an open compact subgroup and $\Cal F$ a family of closed normal
subgroups of $H$. Then there exists $F_0\in\Cal F$ such that
\,$F\cap(F_0K)$ \,is normal in $H$ for each $F\in\Cal F$ with
$F\supseteq F_0$\,.
\end{Lem}
\begin{proof}
For $M$ a subset of $H$ we denote by $\langle M\rangle_s$ the subsemigroup
generated (algebraically) by $M$\,, as before,
$\pi_K\!:H\to H/K$\;(left cosets) shall be the canonical projection.
Since $H$ is compactly generated, $K$ open, there exists a finite subset
$M\subseteq H$ such that $H=\langle MK\rangle_s$\,.
The $K$-orbits in $H/K$ are finite ($H/K$ is discrete, $K$ compact), hence we
can enlarge $M$ (including members of additional cosets) so that
$KM\subseteq MK$ \,(this implies also that
$\langle MK\rangle_s=\langle M\rangle_sK$\,, compare \cite{Mo1}\;L.\,2).
We fix such a finite set $M$ and choose $F_0\in\Cal F$ such that
the cardinality $\lvert\pi_{F_0K}(M)\rvert$ becomes minimal among all
$F\in\Cal F$\,. Then (selecting one element from each coset) we take a
minimal subset $M_0\subseteq M$ such that
$\pi_{F_0K}(M_0)=\pi_{F_0K}(M)$\,. It follows that
$H=\langle M_0F_0K\rangle_s\,,\ KM_0\subseteq M_0F_0K$\,.
\\[1mm]
Consider now $F\in\Cal F$ with $F\supseteq F_0$\,. By the choice
of $F_0\,,M_0$\,, we have
$\lvert\pi_{FK}(M)\rvert=\lvert\pi_{F_0K}(M)\rvert=
\lvert\pi_{F_0K}(M_0)\rvert=\lvert M_0\rvert$\,.
Thus, if $x,x'\in M$ and
$xF_0K\neq x'F_0K$ then $xFK\neq x'FK$ which leads to
$\lvert\pi_{FK}(M_0)\rvert=\lvert M_0\rvert$\,, i.e.,
$\pi_{FK}$ is injective on $M_0$\,. Take $y\in F\cap(F_0K)\,,\ x\in M_0$\,.
Then $y=y_0y_1$ with $y_0\in F_0\,,\ y_1\in K$\,. It follows that
$y_1x=x_1y_2$ with $x_1\in M_0\,,\ y_2\in F_0K$\,. Thus
$\pi_{FK}(x)=\pi_{FK}(yx)=\pi_{FK}(x_1)$ and by injectivity $x_1=x$\,.
We conclude that $yx=y_0xy_2$\,, hence $x^{-1}yx\in F\cap(F_0K)$\,.
For $x\in F_0K$\, obviously $x^{-1}yx\in F\cap(F_0K)$ and by induction
it follows that $x^{-1}yx\in F\cap(F_0K)$ holds for all $x\in H$\,. 
\end{proof}
\begin{proof}[Proof of Theorem\;\ref{th15}]
The theorem follows immediately by combining the result of
M\"oller--Trofimov for compactly generated totally disconnected groups
(\cite{Mo1}\ Th.\,2) with Proposition\;\ref{pr18}. We include an argument
based on Lemma\;\ref{le19}. Let $G$ be a compactly generated totally
disconnected l.c.\,group, $K$ an open compact subgroup.
$\Cal F$ shall denote the family of all compact
normal subgroups, we put $G_1=\bigcup_{F\in\Cal F}F$\,. Then $G_1$ is a
subgroup
of $G$\,. By the observation in the proof of Proposition\;\ref{pr18}, each
$x\in B(G)\cap K$ belongs to some group $F\in\Cal F$\,, hence
$B(G)\cap K=G_1\cap K$\,. By Lemma\;\ref{le19} there exists $F_0\in\Cal F$
such that $F\cap(F_0K)$ \,is normal in $G$ for each $F\in\Cal F$ with
$F\supseteq F_0$\,. Put $F_1=\bigcap_{x\in G}xF_0Kx^{-1}$. Then $F_1$
is a compact normal subgroup and $F\cap K\subseteq F_1$ for each
$F\in\Cal F$\,. Thus $G_1\cap K=F_1\cap K$ is closed, hence (e.g. by
\cite{HR}\;Th.\,5.9) $B(G)$ is closed.
\end{proof}
\begin{Def} \label{def110}	    
Let $G$ be a l.c.\,group. Analogously to \cite{Ro}\;p.\,129 we define the
(topological)
{\it upper $\FC$-\,central series} by $B_0(G)=(e)\,,\;B_1(G)=\overline{B(G)}\,,
\linebreak
B_{\alpha+1}(G)/B_{\alpha}(G)=\overline{B(G/B_{\alpha}(G))}$
for ordinals $\alpha$ and
$B_{\lambda}(G)=\overline{\bigcup_{\alpha<\lambda}B_{\alpha}(G)}$ for limit
ordinals $\lambda$\,. The least ordinal $\alpha$ such that
$B_{\alpha+1}(G)=B_{\alpha}(G)$ is called the {\it $\FC$-\,class} of $G$
\,(or $\FC$-rank as in \cite{FF}). $\zeta_{\FC}(G)=B_{\alpha}(G)$ the
{\it $\FC$-\,hypercentre}. If $\zeta_{\FC}(G)=G$ then $G$ is called
{\it $\FC$-\,hypercentral}.
\\
By induction, it is easy to see that if $(e)=G_0\subseteq G_1\subseteq
\dots$ is any series of closed normal subgroups of $G$ such that
$G_{i+1}/G_i$ is an $\FC_G$\,-\,group for $i\ge0$\,, then
$G_i\subseteq B_i(G)$ for $i\ge0$ (similarly for ordinals).
$G/\zeta_{\FC}(G)$ is the topological analogue of an ICC-group,
i.e. it has no non-trivial relatively compact conjugacy classes.
\end{Def}
\begin{Pro}	\label{pr111}	
Let $G$ be a compactly generated l.c.\,group, $K$ an open compact subgroup
of $G/G^0$. Then there exists $n_0\in\N$ such that
$G^0\cap B_n(G)\subseteq B_{n_0}(G)$ and
$B_n(G)\cap \pi_{G^0}^{-1}(K)\subseteq B_{n_0}(G)$ for all
$n\ge n_0$\,.
\\
$\bigcup_{n\in\N}B_n(G)$ is closed. $B_{n_0}(G)$ is open in $B_n(G)$ for
$n\ge n_0$ and also in $B_{\omega_0}(G)$\,.
\vspace{-1.5mm}
\end{Pro}\noindent
$\omega_0$ denotes the smallest countable ordinal.\\
Thus (combining with Theorem\;\ref{th15}) for a compactly generated
l.c.\,group
\\ $B_{n+1}(G)/B_n(G)=B(G/B_n(G))$ holds for all $n\in\N$ and
$B_{\omega_0}(G)=\bigcup_{n\in\N}B_n(G)$\,.
\begin{proof}
By Theorem\;\ref{th15} the groups $B(G/B_n(G))$ are closed. By induction
one gets from Proposition\;\ref{pr18} that $G^0B_n(G)$ is closed which
implies (\cite{HR}\;Th.\,7.12 and Th.\,5.33) that
\,$(G/B_n(G))^0=(G^0B_n(G))/B_n(G)\cong G^0/(G^0\cap B_n(G))$\,.
\\
As before, we can assume that $G^0$ is a Lie group. Put $H_n=G^0\cap B_n(G)$\,.
This is an increasing sequence of closed normal subgroups. Take $n_1$ such
that $\dim(H_{n_1})$ is maximal. Put $H=Z(G^0/H_{n_1})$ (centre), $N_0$
shall be the connected nil-radical of $G^0/H_{n_1}$\,.
Then for $n\ge n_1\,,\ H_n/H_{n_1}$ is discrete, hence
$H_n/H_{n_1}\subseteq H$\,. $G^0/H_{n_1}$ has no non-trivial connected compact
normal subgroups (these would contribute to $H_{n_1+1}$) and
$H_{n_1+1}/H_{n_1}\cong B(G/B_{n_1}(G))\cap (G/B_{n_1}(G))^0$\,. Hence by
\cite{Ti}\;Th.\,1 (or \cite{Lo3a}\;p.\,10) \;$(H_{n_1+1}/H_{n_1})\cap N_0$
is connected, thus it must be trivial. Since $H^0\subseteq N_0$ it follows
that
$(H_{n_1+1}/H_{n_1})\cap H^0$ is trivial. Similarly, $(H_n/H_{n_1})\cap H^0$
is trivial for all $n\ge n_1$ \,(the connected nil-radical of $G^0/H_n$
is just the image of $N_0$). $H$ is a compactly generated group (e.g. by
\cite{Ho}\;Th.\,XVI.1.2). Hence $H/H^0$ is a finitely generated abelian group
and these satisfy the ascending chain condition for subgroups. It follows
that there exists $n_2$ such that
$G^0\cap B_n(G)=H_n=H_{n_2}=G^0\cap B_{n_2}(G)$ for all $n\ge n_2$\,.
\\
Put $F_n=\pi_{G^0}(B_n(G))\ \:(n\in\N)$. These are closed normal subgroups of
$G/G^0$.
By Lemma\;\ref{le19} there exists $n_3$ such that $F_n\cap (F_{n_3}K)$ is
normal in $G/G^0$ for all $n\ge n_3$\,. Put
$B=B_{n_3}(G)\,,\ F=F_{n_3}=(G^0B)/G^0$\,. Then $(G/B)^0=(G^0B)/B$ and
$(G/B)\big/(G/B)^0\cong G/(G^0B)\cong (G/G^0)\big/F$\,. By
$\pi_{G^0\!B\mspace{-1.5mu},\mspace{1.5mu}G^0}\!:G/G^0\to G/(G^0B)$ we
denote the quotient mapping (identifying $G/(G^0B)$ with
$(G/G^0)\big/((G^0B)/G^0)=(G/G^0)\big/F$\,\,).
We have $\ker(\pi_{G^0\!B\mspace{-1.5mu},\mspace{1.5mu}G^0})=F\,,\
\ker(\pi_{G^0\!B\mspace{-1.5mu},\mspace{1.5mu}B})=(G/B)^0$ and taking care
of the isomorphisms we get
$\pi_{G^0\!B\mspace{-1.5mu},\mspace{1.5mu}B}\circ\pi_B=
\pi_{G^0\!B}=\pi_{G^0\!B\mspace{-1.5mu},\mspace{1.5mu}G^0}\circ\pi_{G^0}$\,.
Put $K_1=\bigcap_{x\in G/(G^0B)}
x\:\pi_{G^0\!B\mspace{-1.5mu},\mspace{1.5mu}G^0}(K)\:x^{-1}$\,.
Then $K_1$ is a compact normal subgroup of $G/(G^0B)$\,, hence
$K_1\subseteq B(G/(G^0B))$ and by the choice of $n_3$ we have for $n\ge n_3$
that
$\pi_{G^0\!B\mspace{-1.5mu},\mspace{1.5mu}G^0}(F_n\cap (FK))=
\pi_{G^0\!B\mspace{-1.5mu},\mspace{1.5mu}G^0}(F_n)\cap K_1$\,.
Furthermore, by Proposition\,\ref{pr16},\linebreak
\,$\pi_B(B_{n_3+1}(G))=B_{n_3+1}(G)/B=B(G/B)$
is open in $\pi_{G^0\!B\mspace{-1.5mu},\mspace{1.5mu}B}^{-1}(B(G/(G^0B)))$
\,(using $(G/B)\big/(G/B)^0\cong G/(G^0B)$\,). Hence
$\pi_{G^0\!B}(B_{n_3+1}(G))$
is open in $B(G/(G^0B))$ and it follows that
$\pi_{G^0\!B}(B_{n_3+1}(G))\cap K_1$ has finite index in $K_1$\,. Now choose
$n_0\ge n_2,n_3$ such that $[\,K_1\,:\,\pi_{G^0\!B}(B_{n_0}(G))\cap K_1\,]$ is
minimal. Then 
\,$\pi_{G^0\!B}(B_n(G))\cap K_1=\linebreak\pi_{G^0\!B}(B_{n_0}(G))\cap K_1$
\,for $n\ge n_0$\,. We have \;$\pi_{G^0\!B}(B_n(G))\cap K_1=
\pi_{G^0\!B\mspace{-1.5mu},\mspace{1.5mu}G^0}(F_n)\cap K_1=\linebreak
\pi_{G^0\!B\mspace{-1.5mu},\mspace{1.5mu}G^0}(\,F_n\cap(FK)\,)=
\pi_{G^0\!B}\bigl(\,(\,G^0B_n(G)\,)\cap (\,B\,\pi_{G^0}^{-1}(K)\,\bigr)$
\,and it follows that
\linebreak
$(G^0B_n(G))\cap \pi_{G^0}^{-1}(K)\subseteq G^0B_{n_0}(G)$\,.
\\
Take \,$x\in B_n(G)\cap \pi_{G^0}^{-1}(K)$\,, then $x=yz$ with
$y\in G^0\,,\ z\in B_{n_0}(G)$\,. Since
$G^0\cap B_n(G)=G^0\cap B_{n_0}(G)$ for $n\ge n_0\ge n_2$
\,and $y\in G^0\cap B_n(G)$\,, we get\linebreak
\,$x\in B_{n_0}(G)$\,. \,Thus
\,$B_n(G)\cap \pi_{G^0}^{-1}(K)=B_{n_0}(G)\cap \pi_{G^0}^{-1}(K)$\,.
Put \,$B_\infty(G)=\bigcup_{n\in\N}B_n(G)$ \,then
\,$B_\infty(G)\cap \pi_{G^0}^{-1}(K)=B_{n_0}(G)\cap \pi_{G^0}^{-1}(K)$
\,and since $\pi_{G^0}^{-1}(K)$ is open in~$G$\,, it follows by
\cite{HR}\;Th.\,5.9
that $B_\infty(G)$ is a closed subgroup. For $n\ge n_0\,,\ B_n(G)$ is open
in $B_\infty(G)$ and also in $B_m(G)$ for $m\ge n$\,.
\end{proof}
\begin{Thm}   \label{th112}	    
Let $G$ be a compactly generated l.c.\,group,
then the $\FC$-\,class of $G$ is at most $\omega_0$\,.
If $G$ is $\FC$-\,hypercentral then the $\FC$-\,class is finite.
\vspace{-1.5mm}
\end{Thm}\noindent
Thus (combined with \cite{Lo2}\;Th.\,1) the compactly generated
$\FC$-\,hypercentral l.c. groups coincide with
the class of compactly generated l.c.\,groups of polynomial growth.
\begin{proof}
The arguments are similar to those in the proof of \cite{FF}\;Th.\,1. Let
$M,B$ be compact subsets of $G$ such that $G=\bigcup_{n\in\N}M^n$ and
the image $\dot B$ in $G/B_{\omega_0}(G)$ is\linebreak $G$-invariant. Put
$B_1=(MBM^{-1}B^{-1})\cap B_{\omega_0}(G)$\,. $B_1$ is compact and
since for $n$~large $B_n(G)$ is open in $B_{\omega_0}(G)$\,, it follows
that there exists $n$ such that\linebreak $B_1\subseteq B_n(G)$\,.
We denote now by $\ddot B$ the image in $G/B_n(G)$ and it follows that
$\ddot x\ddot B\ddot x^{-1}\subseteq\ddot B$ \,for all $x\in M$\,. By
induction, this implies that $\ddot B$ is $G$-invariant, thus
$B\subseteq B_{n+1}(G)$\,. This shows that
\,$B_{\omega_0+1}(G)=B_{\omega_0}(G)$\,.
\\
Now assume that $G$ is $\FC$-\,hypercentral. Then $G=B_{\omega_0}(G)$ and
with $M$ as above, it follows that there exists $n$ such that
$M\subseteq B_n(G)$\,. This implies that $G=B_n(G)$\,.

\end{proof}

\begin{Rem}   \label{rem113}		  
(a) \ A counterexample with a non-compactly generated group where
$B(G)$ is not closed has been given in
\cite{Ti}\;p.\,104. This was extended further in \cite{WY}\;sec.\,6.
\item[(b)] The Theorems\;\ref{th15} and \ref{th112} complete
\cite{Lo2}\;1.4.5. If $G$ is compactly generated and
$\FC$-\,hypercentral it follows from \cite{Lo2}\;L.\,1 that $B(G)$
(and also $B_i(G)\,,\ i\ge0$) is compactly generated. But there are
examples of finitely generated (discrete) groups such that $B(G)$ is
not finitely generated.
\item[(c)] \;For general l.c.\,groups $G$ one can show (extending the type of
arguments given for Proposition\;\ref{pr18}) that $G^0B_1(G)$ is always open
in $\pi_{G^0}^{-1}(B_1(G/G^0))$\,. It follows also that if $B(G)$ is closed,
then $B(G/G^0)$ must be closed.
\end{Rem}
\medskip
\section{Further growth properties} 
\medskip
We give an intrinsic characterization of the (non-connected) nil-radical $N$
(i.e., the maximal nilpotent normal subgroup, \cite{Lo2}\;Prop.\,3)
and the $\FC_G$\,-\,series $(H_n)$ of Theorem\;\ref{th52} in terms of certain
growth properties.

\begin{Def} \label{def61}	    
For a relatively compact symmetric $e$-neighbourhood $V$ generating~$G$
\,(i.e.,
$G = \bigcup\limits_{n = 1}^{\infty} V^n$) put 
\,$\tau_V (x) = \min\,\{n \in \N \cup \{0\} : x \in V^n\}$ for $x \in G$
\,(clearly, we take $V^0 = \{e\}$). For $G$ discrete, this is just the word
length function (see \cite{Ha}\;VI.A.\,Def.1; for the general non-discrete
case see also \cite{FG1}\;Def.\,1.2). The elementary properties of $\tau_V$
are the same as in \cite{Ha}\;VI.A.3(i). For different $V, V'$ the functions
are related as in \cite{Ha}\;IV.B.\,Ex.21\,(iii), i.e., there exist
$c_1, c_2 > 0$
such that \,$c_1\,\tau_V(x) \leq \tau_{V'}(x) \leq c_2\,\tau_V (x)$
\,for all $x$. If $V$ is fixed we will write simply $\tau(x)$.
\\[1mm]
We put
\,$\gamma (x) = \lim\limits_{k\to \infty}
\dfrac{\log \tau (x^k)}{\log k}$ \ for $x \in G$ and call it
\,{\it local growth} at $x$\,.
\\[.8mm]
Furthermore,
$\lVert x\rVert_n = \sup\,\{\,\tau (y x y ^{-1})\!: y \in G
\text{ with } \tau (y)\leq n\,\} \quad (n = 0,1,2,\dots;\linebreak x \in G)$
will be called \,{\it conjugacy operator growth}
(note Remark\;\ref{rem67}\,(d)\,).\vspace{-2mm}
\end{Def}

It will result from the arguments below that for groups of polynomial growth
the limit defining $\gamma$ always exists and either 
$\gamma (x) = 0 \text { or \,} \gamma (x) = \dfrac 1j
\text { for some } j \in \N$. It follows easily from the properties mentioned
above that $\gamma (x)$ does not depend on the choice of $V$.\vspace{3mm}

In the next two Lemmas, we will show that if $x$ belongs to a co-compact
subgroup $H$ of $G$, then $\gamma (x)$ keeps its value when determined
relatively to $H$ \,(compare \cite{Gu}\;L.\,I.6).
This will be used to treat the "generic" case \,$G = \widetilde N \rtimes K$
in Proposition\;\ref{pro64} and then to deduce the general case in
Theorem\;\ref{th65}. In Corollary\;\ref{cor66} we will give some conclusions
on the growth of subgroups.

\begin{Lem}	 \label{le62}	   
Assume that $H$ is a closed subgroup of the locally compact group $G$ such
that $G = CH$ for some compact subset $C$ of $G$. Let $V$ be a
relatively compact $e$-neighbourhood
in $G$ such that \,$H = \bigcup\limits_{n = 1}^{\infty} (V \cap H)^n$. Then
there exists $s\in \N$ such that $V\,C \subseteq  C\,(V \cap H)^s$.
\end{Lem}

\begin{proof}
\,$(C^{-1}VC) \cap H$ being relatively compact, there exists $s \in \N$
such that\linebreak
$(C^{-1}VC) \cap H \subseteq (V\cap H)^s$. Now take $x \in V\,C$.
By assumption, $x = c\,y$ for some $c \in C,\; y \in H$ and it follows that
$y = c^{-1}x \in (C^{-1}VC) \cap H$, giving\linebreak $x \in C (V\cap H)^s$.
\end{proof}

\begin{Lem}   \label{le63}	
Let $G, H, C, V, s$ be as in Lemma\;\ref{le62} and assume that $V$ is
symmetric,
$C \cap H \subseteq V$ and $G = \bigcup\limits_{n = 1}^{\infty}  V^n$ holds.
\ Then\vspace{-.5mm}
$$\tau_V(x)\,\leq\,\tau_{V\cap H}(x)\,\leq\,(s + 1)\,\tau_V(x)\quad
\text{ for all } x \in H\,.$$
Furthermore,\quad $V^n\subseteq C\,(V\cap H)^{s\,n+1}$ \ for all $n$\,.
\end{Lem}\noindent
Recall from \cite{MS} that a co-compact subgroup of a compactly generated
group is always compactly generated. This quarantees the existence of a 
neighbourhood $V$ with the properties needed in Lemma\;\ref{le62}, \ref{le63}. 

\begin{proof}
$\tau_V(x) \leq \tau_{V\cap H}(x)$ \,is always true. For the other part,
we will show by induction on $n \geq 0$ that if
\,$x = x_1\dots x_n\,c \in H$ where  $x_i\in V \ (i = 1, ..., n),\;c \in C$,
then
\,$\tau_{V \cap H}(x) \leq s\,n + 1$. Of course, this will imply
\,$\tau_{V\cap H}(x) \leq s\,\tau_V (x) + 1$ for all $x \in H$ and
then for $x \not= e,\ \,\tau_{V\cap H} (x) \leq (s + 1)\,\tau_V (x)$
\,follows \,(for $x = e$ everything is trivial).
\\
For $n = 0$, we have \,$x = c \in C \cap H \subseteq V \cap H$, \,hence
\,$\tau_{V\cap H} (x) \leq 1$\,. For $n > 0$, we have by Lemma\;\ref{le62},
$x_n c = c'y$
with $c' \in C\,,\; y \in (C \cap H)^s$. Then sub-multiplicativity and the
inductive assumption give 
\\\hspace*{3cm}
$\tau_{V\cap H} (x) \leq \tau_{V\cap H} (x_1\dots x_{n-1} c')\,+\,s
\leq s\,n + 1$\,.
\\[.5mm]
For the final assertion, take $x \in V^n$. Then $x = c\,h$ with
$c \in C,\; h \in H$ and $h^{-1} = x^{-1} c$\,. Consequently, 
$\tau_{V\cap H} (h) = \tau_{V\cap H} (h^{-1}) \leq s\,n + 1$ which gives
\linebreak$x \in C\,(V\cap H)^{s\,n + 1}$.\vspace{1.5mm}
\end{proof}

\begin{Pro}  \label{pro64}	 
Assume that $G = \widetilde N \rtimes K$ where $K$ is compact,
$\widetilde N$ shall be a connected, simply connected nilpotent Lie group.
Put \,$W_0 =
\bigcup \;\{\,\tilde n K \tilde n^{-1} :\tilde n\in\widetilde N\,\}$\,.
\\
{\rm (a)} If the action of $K$ on $\widetilde N$ is faithful, then
$$\widetilde N = \{x \in G:\,\lVert x\rVert_n /n \to 0
\text{ \,for } n \to \infty\,\}\ .$$
{\rm (b)} The following statements are equivalent
\begin{enumerate}
\item[(i)] \ $x \in W_0$ \qquad {\rm (ii)} \ $\gamma (x) = 0$ \qquad
 {\rm (iii)} \ $\{\tau (x^k)\! : k \in \N\}$ is bounded
\item[(iv)] \ $x$ generates a relatively compact subgroup of $G$.\vspace{.5mm}
\end{enumerate}
{\rm (c)} The following statements are equivalent
\begin{enumerate}
\item[(i)] \ $x \in W_0\,C_{j-1}(\widetilde N)$, but
\,$x \notin W_0\,C_j(\widetilde N)$ \qquad
{\rm (ii)} \ $\gamma (x) = \frac 1j$
\item[(iii)] \  there exist  $c_1, c_2 > 0$ such that
$$\qquad c_1\,k^{\frac 1j} \leq \tau (x^k) \leq c_2\,k^{\frac 1j}\qquad
 \text{ for all \ } k \in \N\,.$$
\end{enumerate}
\end{Pro}

\begin{proof}
$(\alpha)$ \,First, we prove (b),\,(c) when $K$ is trivial. This is
essentially well known. Let $\tilde{\fr n}$ be the 
Lie algebra of $\widetilde N$ and decompose
\,$\tilde{\fr n} = \bigoplus\limits_{j = 1}^r\,\fr w_j$ so that
\,$\bigoplus\limits_{i = j}^r\,\fr w_i$ gives the Lie algebra of
\,$C_{j-1}(\widetilde N)\quad (j = 1,\dots,r)$.
Consider any norm on $\tilde{\fr n}$ and put
\,$\varphi (X) = \max\limits_{j=1,\dots,r}\,\lVert X_j\rVert^{\frac 1j}$
\,for \;$X = \sum\limits_{j = 1}^r\,X_j$ \,with $X_j \in \fr w_j$\,.
Then it has been shown in \cite{Gu}\;Proof of L.\,II.1 and Th.\,II.1 that
given a relatively compact $e$-neighbourhood $V$ in 
$\widetilde N$ there exist $c_1, c_2 > 0$ such that
$$\{\exp(X) : \varphi (X) \leq c_1\,n\}\,\subseteq\, V^n\,
\subseteq\,\{\exp (X) : \varphi (X) \leq c_2\,n\}$$
for all $n$ (alternatively, one could use also the method of Tits in the
appendix of \cite{Gr}; \;see also \cite{Br}\;Th.\,3.7, $\varphi$ is one
of the main examples of the homogeneous quasi--norms considered where).
In other words \quad
$\dfrac{\varphi (X)}{c_2} \leq \tau_V (\exp (X)) \leq 
\dfrac{\varphi (X)}{c_1}$\quad for all \,$X \in \tilde{\fr n}$\,.
\\[1mm]
Furthermore, either by inspecting the proof of \cite{Gu}\;L.\,II.1 or by
applying the last formula to $C_j(\widetilde N)$ and using the well known
relation
\,$[C_i(\widetilde N),C_j(\widetilde N)]\subseteq C_{i+j+1}(\widetilde N)$
\,(e.g., by \cite{War}\;Cor.\,1.12) which implies
\,$C_i(C_j(\widetilde N))\subseteq C_{ij+i+j}(\widetilde N)$, it follows
that there exists $c>0$ such that
\,$\tau_{V\cap C_j(\widetilde N)}(x)\ge c\: \tau_V (x)^{j+1}$ \;for all
\,$x\in C_j(\widetilde{N}),\linebreak j=1,2,\dots$\,.
\\
From the first formula, it follows immediately that the property
\;``\,$x \in C_{j-1}(\widetilde N)$ but
$x \notin C_j(\widetilde N)$\;'' is equivalent
to (c)\,(iii). Then (c) is an easy consequence
\,(clearly $W_0 = (e)$) and (b) follows from (c).
\vspace{1mm}
\item[$(\beta)$] \,Now consider the general case for $K$. From
Lemma\;\ref{le63} and $(\alpha)$, the equivalence in (c) follows as 
long as $x \in \widetilde N$. Taking an arbitrary $x = \tilde n\,k \in G$
\,(where
\linebreak$\tilde n \in \widetilde N,\;k \in K$), it follows by induction
on the nilpotency-class of $\widetilde N$ or by \cite{Lo3a}\;L.\,2.19
(taking
$\sigma = \theta = \iota_k$) that it can be written as
$x = y\,w$ with\linebreak
$w \in W_0\,,\ y \in \widetilde N$ such that $y\,w = w\,y$\,.
Clearly, $w$ generates a relatively compact subgroup and
(using $x^k = y^k w^k$) this
implies that $\gamma (x) = \gamma (y)$\,. Thus if \linebreak
$\gamma (x) = 0$\,, we get by
($\alpha$)\,: $\:y = e$\,, i.e., $x \in W_0$ and this proves (b)
\;(the implications 
(i)$\Rightarrow$(iv)$\Rightarrow$(iii)$\Rightarrow$(ii) are easy). If
$y \neq e$\,, the implication (ii)$\Rightarrow$(iii) in (c) follows from
($\alpha$) \,(applied to $y$\,; in particular, this proves also for general
$x$ the existence of the limit
defining $\gamma$ and that $\dfrac 1{\gamma(x)} \in \N$ whenever
$x \notin W_0$). For \,$\gamma (x) = \dfrac 1j\:, \linebreak(\alpha)$ implies
$y \in C_{j-1} (\widetilde N)$\,, hence $x \in W_0\,C_{j-1} (\widetilde N)$\,.
\\[-.3mm]
To get the converse, assume that $x = y' w'$ with
$w' \in W_0\,,\; y' \in C_j (\widetilde N)$. After conjugating, we can
restrict to $w' \in K$. We can also
assume that $KVK = V$. This gives $x^k = y_1\dots y_k w_k$ \,with 
$w_k \in K,\;y_i \in C_j (\widetilde N)$ and
\,$\tau_{V \cap C_j(\widetilde N)}(y_i) =
\tau_{V \cap C_j(\widetilde N)}(y')$. Then  \vspace{-.9mm}
\,$\tau_V(y_1 \dots y_k)^{j+1} \,\leq\,
\dfrac 1c \;\tau_{V \cap C_j (\widetilde N)} (y_1 \dots y_k)\, \leq \,
\dfrac 1c\:k \:\tau_{V \cap C_j (\widetilde N)} (y') $\vspace{.2mm}
\;(by the further properties noted in $(\alpha)$\,)
and this would imply \,$\gamma (x) \leq \dfrac 1{j+1}$\,. \vspace{-.5mm}
Now all the implications in (c) follow.
\item[$(\gamma)$] \,Finally, we prove (a). First assume that
$x\in\widetilde N$.
Again, we may assume $KVK = V$. Then $\tau_V(z)$ depends only on the
$\widetilde N$-component of $z$ and we can reduce to the case where $K$
is trivial. We take up the
notations of $(\alpha)$.
\,$[X,Y]\in\bigoplus\limits_{k = i+j}^r \fr w_k$ \vspace{-2mm}\,holds
for $X \in \fr{w}_i\,,\;Y \in \fr{w}_j$ \,and there
exists $c' > 0$ such that \vspace{-1.5mm}
$$\lVert\,[X,Y]\,\rVert^{\tfrac 1{i+j}} \,\leq\, 
c'\,(\lVert X\rVert\,\lVert Y\rVert)^{\tfrac 1{i+j}}\,=\,
c'\,\varphi (X)^{\tfrac i{i+j}} \; \varphi(Y)^{\tfrac j{i+j}}\,.
\vspace{-1mm}$$
If $X\in\tilde{\fr n}$ is fixed, it follows that there exists $c'' > 0$ such
that \,$\varphi ([X,Y]) \leq c'' \varphi (Y)^{\tfrac{r-1}r}$ whenever
$Y\in\tilde{\fr n}$ and $\varphi (Y) \geq 1$ and then there is $c'''>0$
such that $\varphi (\ad(y)X) \leq c''' \tau (y)^{\tfrac{r-1}r}$ for
$y\neq e$\,. Putting the pieces together, this implies that for every
\,$x \in \widetilde N \quad \lVert x\rVert_n = O (n^{\tfrac{r-1}r})$ holds for
$n \to \infty$.
\\[.2mm]
For the converse, assume that $x = \tilde n\,k$ with
$\tilde n \in \widetilde N,\;k \in K,\;k \not= x$\,. By assumption, $\iota_k$
is non-trivial on $\widetilde N$. Since this is a
semisimple automorphism, it follows easily that the induced automorphism of
$\widetilde N/C_1 (\widetilde N)$ must also be non-trivial.
Take $y \in \widetilde N$ such that
$k y k^{-1} y^{-1} \notin C_1(\widetilde N)$
\;(equivalently, $y k y^{-1} k^{-1} \notin C_1 (\widetilde N)$\,). By
considering the projection to the abelian quotient
$\widetilde N/C_1 (\widetilde N)$, it
follows easily that there exists $c_0 > 0$ such that 
\,$\tau_V (y^n ky^{-n} k^{-1}) \geq c_0\,n$ for all $n$\,. \vspace{.1mm}
Since we have
already shown that \,$\tau_V (y^n \tilde n y^{-n}) = o (n)$\,, it follows that 
\vspace{-1.7mm}\,\,$\inf\limits_n\,\dfrac 1n\,\tau_V (y^n x y^{-n}) > 0$\,, in
particular \,\,$\inf\limits_n \, \dfrac 1n\, \lVert x\rVert_n > 0$\,.
\vspace{1mm}
\end{proof}
\begin{Thm}  \label{th65}	   
Let $G$ be a compactly generated group of polynomial growth having no
non-trivial compact normal subgroups.
\\
{\rm (a)} \;$N = \{x \in G\! : \lVert x\rVert_{n}/n \to 0
\text{ for n} \to \infty\,\}$ \,describes the (non-connected)\linebreak
nil-radical of $G$.
\\
{\rm (b)} \;$H_n = \bigl\{x\in N\!: \gamma (x) \leq \dfrac 1{n + 1}\,\bigr\}
\quad (n = 0,1,\dots)$\vspace{.5mm} are closed normal subgroups
\linebreak of $G$.
They coincide with the groups defined in Theorem\;\ref{th52}\,.
\\[.4mm]
{\rm (c)} For every $x \in G$ we have either
\;$\gamma (x) = 0$ \,or \;$\gamma (x) = \dfrac 1j$ \,for some $j\in\N$\,.
\\
$\gamma (x)\; =\ \;0$ \;holds iff \,$x$ generates a relatively compact
subgroup.\\[.2mm]
$\gamma (x) = \dfrac 1j$ holds
iff there exists \,$c_1, c_2 > 0$ such that 
\;$c_1 k^{\tfrac 1j} \leq \tau (x^k)\leq c_2 k^{\tfrac 1j}$ \linebreak
for all $k\in \N$\,.\vspace{-.3mm}
\end{Thm}
\begin{proof}
Let $\widetilde G = \widetilde N \rtimes K$ be the algebraic hull of $G$
\,(Theorem\;\ref{th2}). For an appropriate $e$-neighbourhood $\widetilde V$ in
$\widetilde G$\,, we put 
$V = \widetilde V \cap G$\,. Then $\tau_V$ and $\tau_{\widetilde V}$ are
comparable
on $G$ as described in Lemma\;\ref{le63}. In particular, they define the same
local growth function $\gamma$ on $G$\,. Now (c) follows from
Proposition\;\ref{pro64}\,(b)\;and\,(c). Combined with Theorem\;\ref{th52},
we get also (b).
\\[.2mm]
To show (a), let $N = \nil(G)$ be the non-connected nil-radical of
$G$\,. By \cite{Lo3a}\;Prop. 4.8\,(c), we have
$N =\widetilde N\cap G$\,.
For $x \in \widetilde G$ put
\,${\lVert x\rVert}_n\sptilde = \sup\,\{\,\tau_{\widetilde V}(y\,x\,y^{-1})\!:
\linebreak y\in\widetilde G\,,\;\tau_{\widetilde V}(y) \leq n\, \}$. Then by 
Proposition\;\ref{pro64}\,(a), $x \in N$ implies
${\lVert x\rVert}_n\sptilde /n \to 0$ \,for $n \to \infty$ and
by Lemma\;\ref{le63}\,,
it follows that $\lVert x\rVert_n /n \to 0$\,. For the converse,
assume that
$x = \tilde n\,k \in G$ with $\tilde n\in\widetilde N,\;k \in K,\;k \neq e$.
Similarly as in
the proof of Proposition\;\ref{pro64} (step $(\gamma)$\,) we will show that
there are $y_n \in N$ such that 
$\tau (y_n) \to \infty$ \,and
\,$\inf\limits_n \dfrac{\tau (y_n x\,y_n^{-1})}{\tau (y_n)} > 0$ which
will complete the proof of (a). 
Replacing $G$ by $G/H_{1}$ (which amounts to replace $\widetilde N$ by
$\widetilde N/[\widetilde N,\widetilde N]$, see Remark\;\ref{rem53}\,(d)\,),
we can assume that
$\widetilde N$ is abelian. By \cite{Lo3a}\;Prop.\,4.8\,(c), $k$ acts
non-trivially on $N$, thus there exists $y \in N$ such that
\,$k\,y\,k^{-1} \neq y$ and
putting \,$y_n = y^n$, the same argument works as in the proof of
Proposition\;\ref{pro64}\,.\vspace{1mm}
\end{proof}
\begin{Cor} \label{cor66}	  
Let $G$ be a compactly generated group of polynomial growth, $H$ a closed
subgroup. If $H$ and $G$ have the same growth then $G/H$ is compact.
\vspace{-1.5mm}
\end{Cor}\noindent
The converse is well known (\cite{Gu}\;Th.\,I.4) and we used it for
Theorem\;\ref{th52}.
\begin{proof}
Factoring the maximal compact normal subgroup, we can assume that $G$ has
no non-trivial compact normal subgroup. Then, passing to the algebraic hull
\,(Theorem\;\ref{th2}), we can assume that \,$G = \widetilde N\rtimes K$ with
$K$ compact, acting faithfully on $\widetilde N$ and $\widetilde N$ connected,
simply connected, nilpotent.
\\
The special case where $H$ is a closed subgroup of $\widetilde N$ can easily
be settled: \;we have $C_n (H) \subseteq C_n (\widetilde N)$ \,for
$n = 0,1,2,\dots$
and then equality of the growth implies
\,(see Remark\;\ref{rem53}\,(a),\,(d)\,)
\;$\rk (H) = \rk (\widetilde N)$. Then $\widetilde N/H$ must be compact
\,(\cite{Ra}\;Th.\,2.10).
\\[2mm]
Proceeding with the general case, let $C$ be the maximal compact normal
subgroup of $H$ \,(\cite{Lo2}\;Prop.\,1). By \cite{Ho}\;Th.\,XV.3.1, we can
assume that
$C \subseteq K$\,. Then normality of $C$ in $H$ implies that
$H \subseteq C_{\widetilde N}(C) \rtimes K$\,. From the special case just
treated, we conclude that $\rk ( C_{\widetilde N}(C))=\rk (\widetilde N)$
hence $C_{\widetilde N} (C) = \widetilde N$ \,(recall that the centralizer
$ C_{\widetilde N}(C)$ is connected, \cite{Lo3a}\;p.\,10), i.e., $C$ acts
trivially on $\widetilde N$ and
by faithfulness, it follows that $C$ must be trivial. Replacing $H$ by its
radical \,(which is co-compact by \cite{Lo2}\;Prop.\,3\;and\;4),
we can assume that $H$ is solvable and
then, after
passing to a subgroup of finite index, we can assume that the projection of
$H$ to
$K$ is contained in an abelian subgroup, i.e., $[H,H]\subseteq\widetilde N$.
\\
Now we consider the special $\FC_H$\,-\,series $(H_n)$ associated to $H$
by Theorem\;\ref{th52}. We write $N_H=\nil(H)$ for the
nil-radical of 
$H,\ \gamma_H$ for the local growth function of $H$, $\gamma_G$~for that of
$G$\,. By Theorem\;\ref{th65}\,, we have 
$H_n = \{x \in N_H : \gamma_H (x) \leq \frac 1{n+1}\}$.\vspace{.5mm} $H < G$
implies \,$\gamma_G \leq \gamma_H$ on $H$. Hence by
Proposition\;\ref{pro64}\,,
$H_n \cap \widetilde N \subseteq C_n (\widetilde N)$. By
Theorem\;\ref{th52}\,, $H_n\cap {[H,H]}^-$ is co-compact in $H_n$\,.
By our construction, 
$[H,H] \subseteq \widetilde N$, thus $H_n\cap\widetilde N$ is co-compact
in $H_n$ and it follows that 
$\rk(H_n) = \rk(H_n \cap \widetilde N) \leq\rk\bigl(C_n(\widetilde N)\bigr)$
for $n = 0,1,\dots$\;. In addition, $\rk (H)\leq \rk (G)$ \ (based on the
definition in
Remark\;\ref{rem53}\,(a), one can prove for arbitrary generalized
$\overline{FC}$-groups that $\rk(H)\leq\rk (G)$ holds whenever $H$ is a
closed subgroup
of $G$\,; but for the present application, the case of solvable groups is
sufficient and this can be done in the standard way). Then by
Theorem\;\ref{th52}
\,(see also the growth-formula in Remark\;\ref{rem53}\,(a)\,), equality of the
growth implies \,$\rk (H_n) = \rk\bigl(C_n(\widetilde N)\bigr)$ for all
$n\geq 1$ and \,$\rk (H) = \rk (G)$.
It follows that $G/H$ is compact \,(again, $\rk(H) = \rk(G)$ implies $G/H$
compact for arbitrary generalized $\overline{FC}$-groups, but the solvable
case is easier).\vspace{0mm plus 1mm}
\end{proof}
\begin{Rems} \label{rem67}	  
(a) \ If $G$ is an arbitrary compactly generated group of polynomial
growth, let $C$ be its maximal compact normal subgroup
\,(\cite{Lo2}\;Prop.\,1).
If $N, H_n$ are defined as in Theorem\;\ref{th65}\,(a),\,(b), then
$N/C$ is the (non-connected) nil-radical of $G/C$ and $(H_n)$ gives the
general version of the normal series mentioned after Theorem\;\ref{th52}\,.
For $G$ totally disconnected, results related to Theorem\;\ref{th65}\,(a)
are shown in \cite{Tr1}
\,(using the language of graph theory; $o$-automorphisms).
\vspace{1.7mm plus 1mm}
\item[(b)] $\widehat H_n =
\bigl\{x\in G\! : \gamma (x)\leq\dfrac 1{n+1}\,\bigr\}\quad (n = 0,1,\dots)$
defines layers in the group $G$\,. Proposition\;\ref{pro64}\,(c)
gives a description in terms of the algebraic hull $\widetilde G$\,.
It follows that $\widehat H_n$ is closed, left and right $H_n$-invariant 
\,($H_n = \widehat H_n \cap N$), it is invariant under all automorphisms of
$G$, $\widehat H_0 = G$, but for $n > 0$ it need not be a group. In fact,
any normal subgroup of $G$ that is contained in $\widehat H_1$ must already
be contained in $H_1$ \,(one can argue as in the proof of
Theorem\;\ref{th65}\,: if
$x \in G\,,\;x \notin N$, there exists $y \in N$ such that
$[x,y] \notin H_1$\,, thus $[x,y] \notin \widehat H_1$).\vspace{0mm plus .5mm}

Fix $n\geq 1$. By Theorem\;\ref{th65}\,(c) and
Proposition\;\ref{pro64}\,(b),\,(c), $x \in \widehat H_n$ holds iff the coset
of $x$ generates a relatively
compact subgroup of $G/H_n$\,. For $G = \widetilde N \rtimes K$ as in
Proposition\;\ref{pro64}\,, it is equivalent that the coset of $x$ in
$G/C_n(\widetilde N)$ belongs to a conjugate of $K$.
\vspace{0mm plus 1mm}
\item[(c)] For $x \in N$, it follows from the proof of
Proposition\;\ref{pro64}\,(a) that $\lVert x\rVert_n$ grows at most like
$n^{\tfrac{r-1}r}$, where $r$
denotes the nilpotency-class of $\widetilde N$ \;(equivalently: $r-1$ is
the largest index for which $H_{r-1}$ is non-trivial).
\vspace{.8mm plus 1mm}
\item[(d)] The conjugacy operator growth functions $\lVert x\rVert_n$ must
not be confused with the notion of conjugacy growth, see e.g.
\cite{Man}\;Ch.\,17, defined for discrete groups and showing a different
asymptotics.
\end{Rems}
\section{Non-compactly generated groups} 
\medskip
\begin{Def} \label{def71}	 
Let $G$ be a l.c.\;group with Haar measure $\lambda$\,. We call $G$ of
{\it bounded polynomial growth} if there exists $d \in \N$ such that for
every compact \linebreak
$e$-neighbourhood~$V$ we have \,$\lambda (V^n) = O (n^{d})$ \,for $n\in \N$\,.
\\[.5mm]
This means that the exponent $d$ can be chosen independently of $V$ \,(of
course, $\sup\limits_{n\in\N} \dfrac{\lambda (V^n)}{n^d}$ \,will depend on
$V$).\vspace{.3mm} We will show \,(Theorem\;\ref{th73},\,\ref{th74})
that for this class
of not necessarily compactly generated groups $G$ similar structure theorems
hold as in the compactly generated case.
\end{Def}

\begin{Pro}  \label{pro72}	   
Let $G$ be a compactly generated group of polynomial growth, $H$ shall be a
closed subgroup having the same growth as $G$\,. We write $C$ for the maximal
compact normal subgroup of $G$\,. Then the following holds:
\\
{\rm (i)} \ $C\cap H$ is the maximal compact normal subgroup of $H$.
\\
{\rm (ii)} \ If $C$ is trivial, then \,$\nil (H) = \nil (G) \cap H$.
\\
{\rm (iii)} \ If $C$ is trivial, let \,$\widetilde G = \widetilde N\rtimes K$
be the algebraic hull of $G,\ K_H$ shall be the closure of the image of
$H$ in $K\ (\cong \widetilde G/\widetilde N)$. Then
$\widetilde N \rtimes K_H$ is
the algebraic hull of~$H$. In particular, $H$ and $G$ have the same connected 
nil-shadow $\widetilde N$.
\end{Pro}

\begin{proof}
By Corollary\;\ref{cor66} \,$H$ has the same growth as $G$ iff $G/H$ is
compact. Let \linebreak
$\widetilde N\rtimes K$ be the algebraic hull of $G/C$ (Theorem\;\ref{th2})
and define $K_H$ as in (iii). Then $H C/C$ (the image of $H$ in $G/C$) is
contained in $\widetilde N \rtimes K_H$ and co-compact. By
\cite{Lo3a}\;Prop.\,4.8\,(c), $H C/C$ has no non-trivial compact
normal subgroup.
By \cite{HR}\;Th.\,5.33, $H C/C \cong H/(C \cap H)$ and (i) follows.
\\
If $C$ is trivial, then by \cite{Lo3a}\;Prop.\,4.8\,(c) \,
$\nil (H) = H \cap \widetilde N = H \cap G \cap \widetilde N =
H \cap \nil (G)$ and (ii) follows.
\\[-.3mm]
Finally, by \cite{Lo3a}\;Th.\,3, $\widetilde N\rtimes K_H$ is (isomorphic to)
the algebraic hull of $H$, giving (iii).\vspace{-1.2mm}
\end{proof}

\begin{Thm}  \label{th73}
Let $G$ be a l.c.\;group of bounded polynomial growth. Then there exists a
closed normal subgroup $C$ with the following proporties:\vspace{-.8mm}
\begin{enumerate}
\item[(i)] \ $C$ is the directed union of compact open groups \
(equivalently: every finitely generated subgroup of $C$ is relatively
compact).
\item[(ii)] \ $G/C$ is a Lie group.\vspace{-.8mm}
\end{enumerate}
If $H$ is a closed compactly generated subgroup of $G$, then $C \cap H$ is
compact. If $C$ is chosen maximal and $H$ has maximal growth (among the
compactly generated subgroups of $G$), then $C \cap H$ is the maximal
compact normal subgroup of $H$.\vspace{-.7mm}
\end{Thm}
\begin{proof}
For a closed compactly generated (c.c.g.) subgroup $H$ of $G$ let $C_H$ be
its maximal compact normal subgroup \,(\cite{Lo2}\;Prop.\,1). Let $d$ be the 
supremum of the growths of the c.c.g.\;subgroups of $G$. Bounded polynomial
growth and \cite{Gu}\;Th.\,I.2 imply that $d$ is finite. Fix $H$ of growth
$d$\,. Then
by Proposition\;\ref{pro72}\,(i) \,$C_H = C_{H'} \cap H$ for every
c.c.g.\;subgroup $H' \supseteq H$. Put $C = \bigcup\{C_{H'}\!: H'
\text{ c.c.g.\;subgroup, }
H' \supseteq H\,\}$. To $x\in G$ and $H'$ we can find (taking an appropriate
generator) a c.c.g.\;subgroup $H'' \supseteq H'\cup\{x\}$, hence
$H'' \supseteq xH'x^{-1}$. Then it is easy to see that $C$ is a normal
subgroup of
$G$. Choose now $H'$ open with $H' \supseteq H$\,. Then $C \cap H' = C_{H'}$
compact implies that $C$ is closed and since $H'/C_{H'}$ is a Lie group
\,(\cite{Lo1}\;Th.\,2), it follows that $G/C$ is a Lie group.
\\
Finally, to show the equivalent condition in (i), it suffices to prove
that if $B$ is an almost connected l.c.\;group such that every
finitely generated subgroup is relatively compact, then $B$ is compact.
Factoring some compact normal subgroup, we may assume that $B$ is a
Lie group. It follows from \cite{Va}\;Th.\,2.10.1
that $B^0$ is topologically finitely  generated and we conclude that $B$
is compact.\vspace{1.5mm plus 1mm}
\end{proof}

\begin{Thm}  \label{th74}     
Let $G$ be a l.c.\;group of bounded polynomial growth and assume that $G$
has no non-trivial subgroup $C$ as in Theorem\;\ref{th73}\,. then the
following holds:
\\[1mm]
{\rm (i)} \ There exists a semidirect product
\,$\widetilde G = \widetilde N \rtimes K$ such that $K$ is compact,
$\widetilde N$ is a connected, simply connected,
nilpotent Lie group, $K$ acts faithfully on $\widetilde N$, and a
continuous injective homomorphism \,$j\!: G \to \widetilde G$ such
that $j(H)$ is closed for every closed compactly generated subgroup $H$
of $G$\,.
\\[.5mm]
{\rm (ii)} \ There exists a maximal nilpotent normal subgroup $N$ and
a maximal solvable normal subgroup $R$ in $G$\,. \,$G/R$ is compact, $R/N$ has
an abelian subgroup of finite index.\vspace{-2mm}
\end{Thm}\noindent
This extends the result for discrete groups in
\cite{Wi}\;Th.\,3.1. A quantitative refinement in the discrete
case is \cite{Man}\;Th.\,9.10.\vspace{-1.5mm}
\begin{proof}
($\alpha$) \,Fix $H$ as before.
Put $N = \bigcup\{\nil(H')\!: H'\text{ c.c.g.\;subgroup, } H'\supseteq H\,\}$.
Then (arguing as before) Proposition\;\ref{pro72}\,(ii) implies that $N$ is
the maximal
nilpotent normal subgroup of $G$. As in the compactly generated case, we call
it the (non-connected) nil-radical of $G$\,.
\\
We can assume that $H$ is open in $G$\,. Triviality of $C$ implies
(Proposition\;\ref{pro72}\,(i)) that $H$ has no non-trivial compact
normal subgroup (the same for all $H'$ as above). Let
$\widetilde H = \widetilde N\rtimes K_H$ be a fixed algebraic hull of $H$.
We will consider $H$ as a subgroup of 
$\widetilde H$. Next, we want to construct an embedding of $N$ into
$\widetilde N$ and an action of $G$ on $\widetilde N$.
\\[.5mm]
Let $H'$ be a c.c.g.\;subgroup with $H'\supseteq H$ and $\widetilde H'$ any
algebraic hull of $H'$ with embedding \,$j_{H'}\!: H' \to \widetilde H'$.
By \cite{Lo3a}\;Th.\,3, there exists a uniquely determined continuous
homomorphism $\phi\!: \widetilde H \to \widetilde H'$ such that
$\phi|H = j_{H'}|H$\,.
Proposition\;\ref{pro72}\,(iii) \,(and uniqueness in \cite{Lo3a}\;Th.\,3)
implies that $\phi (\widetilde N) = \nil (\widetilde H')$, hence
\,$j = \phi^{-1} \circ j_{H'}|\nil (H')$
is defined and by routine calculation it follows from \cite{Lo3a}\;Th.\,3
that this mapping does not depend on the choice of the algebraic hull
$\widetilde H'$.
Then it is easy to see that this gives a continuous homomorphism
$j\!: N \to \widetilde N$ which is the identity on $\nil (H)$.
In a similar way, we define
for $x \in \widetilde H'\,,\;n \in \widetilde N$\,:
$x \circ n = \phi^{-1}\bigl(j_{H'}(x)\,\phi(n)\,j_{H'}(x)^{-1}\bigr)$ and get
a continuous action of $G$ on $\widetilde N$. For
$x \in N$, we have \,$x \circ n = j(x)\,n\,j(x)^{-1}$. The action of $G$ on
$\widetilde N$ induces a continuous homomorphism 
\,$\alpha\!: G \to \Aut(\widetilde N/[\widetilde N,\widetilde N])$.
By \cite{Lo3a}\;Prop.\,3.3 \,(applied to $H'$), we have
$\ker \alpha = N$.
Similarly, there
is an action of $\widetilde H$ on $\widetilde N$ and on
$\widetilde N/[\widetilde N, \widetilde N]$\,. $\nil(\widetilde H')$ acts
trivially on $\widetilde N/[\widetilde N, \widetilde N]$,
hence $\alpha (H')$ is relatively compact for all c.c.g.\;subgroups $H'$ of
$G$ \,(in particular, all finitely generated subgroups of $\alpha (G)$
are relatively
compact). Then it follows from \cite{Me} \,(one can also use
an argument as in the proof of \cite{Ra}\;Th.\,8.31) that $\alpha (G)$ is
relatively compact.
\vspace{.8mm plus 1mm}
\item[($\beta$)] \,To show the remaining properties of (ii),
we use a similar
strategy as in \cite{Lo2}. Let $R^0$ be the connected radical of the Lie
group $G$. By \cite{Lo1}\;L.\,9 (and its proof, applied to $G/R^0$ and
using that $G^0/R^0$ is compact by \cite{Lo2}\;Prop.\,4) there exists a closed
normal subgroup $D$ of $G$ such that $G/D$ is compact and $D^0 = R^0$. Let
$E$ be a c.c.g.\;subgroup of $D,\ R_E$ the radical of $E$\,. Then
$R_E \supseteq D^0 \cap E$ is open, hence by \cite{Lo2}\;Prop.\,4, $R_E$
has finite index in $E$\,.
It follows that $\alpha (E)$ has an abelian subgroup of finite index.
\\[.5mm plus .5mm]
Put $\Cal D = \overline{\alpha(D)}$\,. By \cite{Ra}\;Cor.\,8.33
(\,$\Cal D^0/Z (\Cal D^0)$ is semisimple, hence there exists a finitely
generated subgroup $E_0$ of $D$ such that the image of $\alpha (E_0)$ is
dense in $\Cal D/Z (\Cal D^0)$\,; taking $E=\overline{E_0}$ we see that
$\Cal D/Z (\Cal D^0)$ is finite), it follows that $\Cal D^0$ is abelian. Thus
$D_0 = \alpha^{-1} (\Cal D^0)$ is a solvable normal subgroup of $G$ and
$D_0/N\ \,(\cong \alpha (D_0)\,)$ is abelian.
$D/(D_0 \cap D) \cong DD_0/D_0 \cong \Cal D/\Cal D^0$ is finite, hence
$G/D_0$ is compact. Since $(G/D)^0$ is semisimple, we conclude that
$(G/D_0)^0$ is semisimple. Let $R$ be the subgroup of $G$ such that $R/D_0$
is the radical of $G/D_0$ \,(\cite{Lo2}\;Prop.\,3). Then $R$ is the maximal
solvable normal subgroup of $G$ and $R/D_0$ is finite. This implies the
properties stated in (ii). As before, we call $R$ the radical of $G$\,.
\vspace{2mm plus1mm}
\item[($\gamma$)] \,For the construction of an embedding as in (i), we will
first show that we may assume that $N$ is connected. Openness of $H$ implies
(Corollary\;2.6) that\linebreak 
$|H'/H| < \infty$ for every c.c.g.\,subgroup $H' \supseteq H$\,. Consequently,
$\nil (H) = \nil (H') \cap H$ has finite index in $\nil (H')$\,. This gives
(see \cite{Lo3a}\,1.2)
\,$j(\nil (H'))\subseteq \nil (H)_{\R}$ for all $H'$ as above and it follows
that $\nil (H)_{\R}$ is a $G$-invariant subgroup of $\widetilde N$
containing $j(N)$\,.
\\[0mm plus .5mm]
Now we apply \cite{Lo3a}\;Prop.\,1.4 and get
(by ``pasting'' with $\nil(H)_{\R}$\,) a continuous embedding of
$G$ into a l.c.\,group $G_{\R}$ 
\,(in general, $j(N)$ will not be closed in $\nil(H)_{\R}$ and then
the image of $G$ will not be closed in $G_{\R}$\,). The proof of
\cite{Lo3a}\;Prop.\,1.4
shows that \,$H' \nil (H)_{\R}$ \,is open in $G_{\R}$ and this coincides
(also topologically) with the group $H'_{\R}$ defined at the end of
\cite{Lo3a}\,2.1, in particular
$H'$ is closed in $G_{\R}$ for every c.c.g.\;subgroup $H'$ of $G$\,.
Furthermore,
\,$G_{\R} = \bigcup\,\{H'_{\R} : H' \text{ c.c.g.\;subgroup, }
\linebreak H'\supseteq H\}$\,.
$H'_{\R}$ has the same growth as the co-compact subgroup $H'$ and every
compact subset of $G_{\R}$ is contained in some $H'_{\R}$\,. It follows that
$G_{\R}$ has bounded polynomial growth, $\nil(H)_{\R}$ is the nil-radical of
$G_{\R}$ (e.g., by \cite{Lo3a}\;Cor.\,3.5) and $G_{\R}$ has no
non-trivial subgroup $C$
as in Theorem\;\ref{th73} \,(since again by \cite{Lo3a}\;Cor.\,3.5
\,$H'_{\R}$ has no non-trivial compact normal subgroup).
\vspace{.5mm plus 1mm}
\item[($\delta$)] \,If $N$ is connected, we can proceed as in the proof of
Theorem\;\ref{th2} (in \cite{Lo3a}). We just indicate the steps. By (ii),
there exists a closed normal
subgroup $G_1$ of $G$ such that $G/G_1$ is compact and $G_1/N$ is abelian.
Then (using connectedness of $N$), $(G_1,N)$ satisfies the assumptions of
\cite{Lo3a}\,2.1 \,(although $G_1$ need not be compactly generated). Choose
$\Cal C$ as in \cite{Lo3a}\;Prop.\,2.15 and put
$K_1 = \overline{\Cal C}$\,. Then
$K_1$ is compact \,(as before, it suffices to consider the restrictions of the
automorphisms to $N$ and then this follows from relative compactness of
$\alpha(G)$ which was shown in $(\alpha)$\,).
Define $G_2 = G_1 \rtimes K_1$ and use
\cite{Lo3a}\;Prop.\,1.4 to get a common extension $G_3$ of $G$ and
$G_2$\,.
Then $G_3/G$ and $G_3/N_3$ are compact, where $N_3$ denotes the nil-radical
of $G_3$\,. Finally, put $G_4 = G_3/P_3$ where $P_3$ is the
maximal union of compact subgroups as in Theorem\;\ref{th73}\,(i), and then
as in $(\gamma)$ we consider $\widetilde G = (G_4)_{\R}$ which satisfies the 
required properties.\vspace{1mm plus 1mm} 
\end{proof}

\begin{Ex}	\label{ex75}
Take $\gamma \in \R\setminus\{0\}$ and consider $G = \C \rtimes \Q$ with the
action 
$q \circ z = e^{2\pi i\gamma q} z$ for $q\in\Q\,,\;z\in\C$
\,($\Q$ equipped with the discrete topology) -- compare
\cite{Lo3a}\;Ex.\,4.12\,(a).
Then $\widetilde G = (\C \times \R) \rtimes K$ with
$K = \{\beta \in \C : |\beta| = 1\},\linebreak \beta\circ(z,t) = (\beta z,t)$ and
embedding
\,$j(z,q) = (z,q,e^{2 \pi i\gamma q})$. Here $j(G)$ is not closed, 
${j(G)}^{-} = \{(z, t, e^{2\pi i\gamma t}) : z \in \C,\,t \in \R\}\,\cong
\C\rtimes \R$ \;(with standard topology on $\R$).
For $\gamma$ irrational, we have
$N = \C$ and for $\gamma \in \Q \setminus \{0\}\,,\
N = \C\times \frac1{\gamma}\,\Z$\,.
\end{Ex}
\vfill
\section{Weighted group algebras} 
\medskip
\begin{varrem}\label{di81}		
If $G$ is a l.c.\;group, a function $\omega\!: G \to [1, \infty]$ is
called a {\it weight}, if it is bounded on compact subsets,
upper semicontinuous (\cite{HR}\;Def.\,11.9) and satisfies
$\omega (xy) \leq \omega(x)\,\omega (y) \text{ for all } x, y \in G$\,.
Moreover, we will always assume that $\omega$
is {\it symmetric}, i.e.,
$\omega (x^{-1}) = \omega (x)\text{ \,for all } x \in G$ \,(see \cite{FG1}
for further comments and examples). $\omega$ is said to satisfy
the Gelfand-Naimark-Raikov (GNR) condition at $x$ if
\,$\lim\limits_{k \to\infty} \omega (x^k)^{\frac1k} = 1$ \;\vspace{.1mm}
(in \cite{FG2} this is called
GRS, Gelfand-Raikov-Shilov condition). For a weight $\omega$\,,
\,$L^1(G,\omega)$ denotes the corresponding weighted $L^1$-space with norm
$\int\limits_G |f(x)|\,\omega (x)\,dx$\,. This is a Banach algebra with
ordinary convolution as multiplication. Since $\omega$ is symmetric,
the standard involution is defined on $L^1(G,\omega)$.\vspace{2mm plus.5mm}

Assume that $G$ is compactly generated of polynomial growth. Then it was
shown in \cite{Lo2}\;Cor.\,1 that $L^1(G)$ is a symmetric Banach algebra.
In \cite{FG1} this
question was investigated for $L^1(G,\omega)$ and (using almost the same
technique)
it was completed in \cite{FG2} (and presented in a better readable manner).
It turned out that the GNR (or GRS)
condition is necessary and sufficient for symmetry of $L^1(G,\omega)$.
Necessity results from Gaussian estimates
for random walks due to Hebisch and Saloff-Coste (extending earlier work of
Varopoulos). The proof of sufficiency adapts the original method of Ludwig
that was used for ordinary $L^1(G)$. Furthermore, it was shown in
\cite{FG2}\;Th.\,3.3 that under the assumption of a uniform analogue of
the GNR-condition (called condition (S), see below) symmetry of
$L^1(G,\omega)$ can be deduced from symmetry of $L^1(G)$ more directly.
It was noted there that
for nilpotent groups there is also a direct argument for
equivalence of the two conditions. Using our embedding,
Theorem\;\ref{th2}, we will show
below that this can be seen in general (although this does not give a
shorter argument in the general case).\vspace{1.2mm plus .5mm}

Let $C$ be the maximal compact normal subgroup of $G$, \;$N/C = \nil (G/C)$
and $(H_n)$ the normal series defined by Theorem\;\ref{th65}\,(b) and let
$L_1$ be a closed subgroup of $G$ such that $L_1 N$ is closed and co-compact
in $G$, \,$L_1\supseteq H_1$ and $[L_1,L_1] \subseteq H_1$ \,(see
\cite{Lo3a}\;Cor.\,4.10\,(d) applied to $G/H_1$ and
Remark\;\ref{rem53}\,(d); one can take the group $L_1$ of
\cite{Lo3a}\;Cor.\,4.9, see also the proof of Theorem\;\ref{th52}). Let
$B_0 = \{y_1, ..., y_{m_1}\}$ be a finite
subset of $L_1$\,, generating a co-compact subgroup of the abelian group
$L_1/H_1$\,. $B_i = \{y_{m_i+1}, ..., y_{m_{i+1}}\}$ shall be a finite subset
of $H_{i-1}$\,, generating a co-compact subgroup of $H_{i-1}/H_i$ \;(by
Theorem\;\ref{th52}\,(ii), $H_{i-1}/H_i$ is abelian and compactly generated).
Put $B = B_0 \cup B_1$\,.
\end{varrem}
\begin{Thm}[\cite{FG2}\;Th.\,1.3]  \label{th82}        
Let $G$ be a compactly generated group of polynomial growth, $\omega$ a
symmetric weight on $G$\,. Then the following statements are
equivalent.
\begin{enumerate}
\item[(i)] \ $L^1(G,\omega)$ is a symmetric Banach algebra.
\item[(ii)] \ $\omega$ satisfies the GNR-condition on $G$\,.
\item[(iii)] \ $\omega$ satisfies the GNR-condition for all $x\in B$,
where $B$ is a finite subset of $G$ as defined above.
\end{enumerate}\vspace{-2.5mm}
\end{Thm}\noindent
The GNR condition means that $\omega$ should grow non-exponentially along
cyclic subgroups. (iii) says that this growth limit holds automatically on
$H_1$ if it holds on a sufficiently large finite set of elements outside
$H_1$\,. Recall that $H_1$ consists of those elements of $N$ where the local
growth is smaller than $1 \ (\,\lim \frac{\tau (x^k)}k = 0)$. 

\begin{proof}
($\alpha$) \,Let $V$ be any compact symmetric $e$-neighbourhood
generating $G$\,. It was shown in \cite{FG2}\;Th.\,1.3 (see also 5.1)
that $L^1 (G,\omega)$ is symmetric iff (ii) holds and that this is
also equivalent to
$\lim\limits_{k \to\infty}\,\bigl(\sup\,\{\omega(x): x\in V^k\}\bigr)^{1/k}
= 1$ \;(they call this condition (S)\,).
In particular, (iii) is necessary for symmetry. We will show
that (iii) \,(hence also (ii)) implies condition (S). Let $r$ be the minimal
index for which $H_r = C$\,. If $r = 0$ \,(i.e., $N = C$), then $G$ is
compact and everything is trivial.
\\[1mm]
For \;$\mathbf t_i= (t_{m_i +1},\dots,t_{m_{i+1}}) \in \Z^{m_{i+1} - m_i}
\ (m_0 = 0)$, we write
\,$\mathbf y_i^{\mathbf t_i} =
y_{m_i + 1}^{t_{m_i+1}}\dots y_{m_{i+1}}^{t_{m_{i+1}}}$\,.
\vspace{2mm}
\item[($\beta$)] \,We claim that there exists a compact subset $C_0$ of $G$
and $s\geq 1$ such that\vspace{-1mm}
\begin{align*} V^k\;\subseteq\
\bigl\{\,\mathbf y_0^{\mathbf t_0} \dots \mathbf y_r^{\mathbf t_r} y \,:\
y \in C_0\,,\ &|t_j| \leq s\,k \text{ \;for } 0 < j \leq m_1 \text{ \ and }
\\[-.5mm]
&|t_j| \leq s\,k^i \text{ \;for } m_i < j \leq m_{i+1}\,,\ i > 0\,\bigr\}
\end{align*}\vskip-1.5mm\noindent
holds for all $k$ \ (with \,$C_0\,,\,s$ \,not depending on $k$\,, but of
course they will depend on $V$ and $B_i$\,; it is easy to see that there is
a corresponding opposite inclusion).
\\[1mm]
For $r = 1$, the claim follows immediately from the definition of
$B_0\,,\,B_1$
\,(similarly as below, one can first consider the projection to $G/N$ and
then the remaining contribution from $N$\,). Using induction with respect
to $r$, we assume that the claim holds for $r - 1$ and \,(passing to $G/C$)
that $H_r$
is trivial. Applying the hypothesis to the image of $V$ in $G/H_{r-1}$\,, we
get a compact subset $C_0\subseteq G$ and $s \geq 1$ so that every
$x \in V^k$ can be written as
\,$x = \mathbf y_0^{\mathbf t_0}\dots\mathbf y_{r-1}^{\mathbf t_{r-1}} y\,z$
\,with $y \in C_0\,,\;z\in H_{r-1}$\,, $t_j$ as above. By
Theorem\;\ref{th65}\,(c), there exist $c_i > 0$ such that
$\tau_V (y_j^k) \leq c_i k^{\frac 1i}$ \,for
$m_i < j \leq m_{i+1}\,,\;0 < i < r$\,, similarly for $i = 0$. Put
\,$c = \sup\limits_{y' \in C_0} \tau_V (y')$\,, then
\,$\tau_V(y\,z\,y^{-1})\leq c +k+
k\,s\,\bigl(m_1 c_0 + \dots+(m_r - m_{r-1})\,c_{r-1}\bigr)$\,.
Since $H_r$ is 
trivial, it follows from the formulas given in step (a) of the proof of
Proposition\;\ref{pro64} that there exists $d > 0$ such that
\,$\tau_{V\cap H_{r-1}}(u)\leq d\,\tau_V(u)^r$ \,for all $u\in H_{r-1}$\,. By
the construction of $B_r$\,, there exists a compact subset $C'$ in $H_{r-1}$
and $s' > 0$ such that every \,$u \in (V \cap H_{r-1})^k$ can be written
as \,$u = \mathbf y_r^{\mathbf t_r} y'$ with $y' \in C',\;|t_j| \leq s'k
\text{ \,for } m_r < j \leq m_{r+1}$\,.
Putting this together gives a representation of \,$y z y^{-1}$ with bounds
as claimed above (when enlarging \,$C_0\mspace{1mu},\mspace{.5mu}s$
\,appropriately), proving the
induction step.
\vspace{.7mm}
\item[($\gamma$)] \,Next, we claim that there exists a compact subset $C_1$
in $G$ and an integer $t > 0$ such that for all $k$\,, we have
\vspace{-2mm plus .3mm}
$$V^k\,\subseteq\,\{\,y_{i_1}^{u_1}\dots y_{i_t}^{u_t} y :\
y \in C_1\,,\,|u_j| \leq t\,k\,,\,0 < i_j \leq m_2
\text{ \,for } j = 1,\dots, t\,\}.\vspace{-1.8mm}$$
Thus the elements of $V^k$ can be written as products of powers from $B$ with
a uniformly bounded number of factors and the exponents bounded by a multiple
of~$k$ \,(i.e., $t\,,\,C_1$ are independent of $k$\,, the $i_j$ need not be
increasing). Property~$(S)$ follows
immediately from the claim if (iii) holds.  For $r = 1$ the claim is just a
restatement of that in $(\beta)$. Again, we use induction on $r$, assuming
that
the claim holds for $r-1$ and that $H_r$ is trivial. Then, as in $(\beta)$,
we apply the hypothesis to $G/H_{r-1}$, obtaining a representation as claimed, 
up to a remainder belonging to $H_{r-1}$. This remainder can be estimated
as in $(\beta)$ and it remains to be shown that every power
$y_j^{t_j} \text{ with } |t_j| \leq s k^r,\,m_r < j \leq m_{r+1}$ has a
product representation as claimed above. This will be done by a special
choice of the set $B_r$ \,(in fact, since $H_{r-1}$ is abelian, it is not
hard to see that this gives corresponding representations for any other choice
of $B_r$).
\\
We start with the case $r = 2$ \,(which is slightly more complicated). Then
$H_2$ is trivial, hence by Theorem\;\ref{th52}\,(ii), 
$[\widetilde N,\widetilde N]\supseteq H_1$ is central in
$\widetilde N$. Now, we argue as in Remark\;\ref{rem53}\,(c). Put 
$D = [\widetilde N,N] \cap N$. It is not hard to see that
$D \subseteq H_1$ is closed, $N/D$ torsion free abelian. We choose $B_2$
by first selecting a set $B'$ covering $H_1/D$ and then a contribution $B''$
for $D$\,.
\\
It will be no restriction to assume that $G = NL_1$ \,(see also
Lemma\;\ref{le63}\,). Let $\widetilde N \rtimes K$ be the algebraic hull,
then 
$[L_1, L_1] \subseteq N$ implies that $K$ is abelian and \,(replacing $L_1$ by
a subgroup of finite index) we may assume that $K$ is connected. Recall
\,(Remark\;\ref{rem53}\,(d)\,) that $N/H_1 = \nil (G/H_1)$. This implies
$N/D = \nil (G/D)$. Since $N^0\triangleleft \widetilde N$
\,(\cite{Lo3a}\;Prop.\,4.4\,(c)\,),
we have $[\widetilde N,N^0] \subseteq D$, hence $(N/D)^0 = N^0D/D$ \,(being a
quotient of $N^0/[\widetilde N, N^0]$) is an $\FC_G$\,-\,group. Then \,(see
the comment
after \cite{Lo3a}\;Cor.\,4.10) $N/D$ is an  $\FC_G$\,-\,group. By
assumption, $[L_1, L_1] \subseteq H_1$\,. Then it follows from
\cite{Lo3a}\;Cor.\,4.10\,(d)
that $G/([L_1,L_1]D)^-$ has abelian growth, hence \,(by minimality of
$H_1$), the image of \,$[L_1, L_1]$ in $H_1/D$ must be co-compact. Thus
we may choose $B'$ by taking
$[y_{j_1}, y_{j_2}]\,,\,1 \leq j_1 < j_2\leq m_1$\,.
By \cite{Lo3a}\;Prop.\,4.4\,(a), we have
$\widetilde N = N^0\widetilde M $, hence
$[\widetilde N, \widetilde N] =
[\widetilde M, \widetilde M]\,[\widetilde N, N^0]$\,. Thus $G$ acts
trivially on \,$[\widetilde N, \widetilde N]/[\widetilde N, N^0]$ and it
follows that $H_1/D$ is contained in the centre of $G/D$. This implies that 
$[y_{j_1}, y_{j_2}]^{kl} = [y_{j_1}^k, y_{j_2}^l]\!\mod D \text{ \,for all }
k, l \in \Z$ and then the desired product representation \,(with a 
remainder in $D$) of powers from $B'$ can be obtained.
\\[0mm plus .5mm]
We have $[\widetilde N, N] = [\widetilde M, N]\,[N, N]$\,. For
$x \in G\,,\,z \in \widetilde N$, we write
\,$[x, z]_u =\linebreak \iota(x)_u(z)\,z^{-1}$
\,(where \,$\iota(x)_u$ denotes the unipotent part of the Jordan decomposition
[\,\cite{Lo3a} 2.1]; ordinary commutators are defined as $[x,y]=xyx^{-1}y^{-1}$
as in \cite{Lo3a}\;p.\,9).
If \,$x = x_2 x_1$ with
$x_1\in K\,,\, x_2 \in \widetilde M$, one has by
\cite{Lo3a}\;Cor.\,4.5
\,$[x, z]_u = [x_2, z]$\,.
By \cite{Lo3a}\;Prop.\,4.4\,(a),
$(GK) \cap \widetilde M$ is co-compact in $\widetilde M$ and it follows
that the image of \,$[L_1, N]_u$ is co-compact in $D/[N, N]$. Thus we may
choose $B''$
by taking \,$[y_{j_1}, y_{j_2}]_u$ \,for $1 \leq j_1 \leq m_1 < j_2 \leq m_2$
\,and \,$[y_{j_1}, y_{j_2}]$ \,for \,$m_1 < j_1 < j_2 \leq m_2$\,. As
before, we have 
\,$[y_{j_1}\,, y_{j_2}\,]_u^{kl} = [y_{j_1}^k\,, y_{j_2}^l\,]_u 
\text{ \,for }
1\! \leq\! j_1\! \leq\! m_1\! < j_2\! \leq\! m_2$
\,(since \,$[N,\widetilde N]$ is connected and an
$\FC_{\widetilde G}$\,-\,group,
it follows from \cite{Lo3a}\,2.5 that $\iota(x)_u$ is the identity on
\,$[N, \widetilde N] \supseteq D \text{ \,for all } x \in G$\,)
\,and also for
$m_1 < j_1 < j_2 \leq m_2$\,. This gives the product
representation for powers from $B''$.
\\[.1mm]
For $r > 2,\ [L_1, H_{r-2}]_u \cup [H_{r-2}, H_{r-2}]$ generate a
co-compact subgroup of $H_{r-1}$\,. As above, we can choose
$B_r =\{\,[y_{j_1}, y_{j_2}]_u\!: 1 \leq j_1 \leq m_1\,,\,
m_{r-1} < j_2 \leq m_r\} \cup \linebreak\{\,[y_{j_1}, y_{j_2}] :
\,m_{r-1} < j_1 <j_2 \leq m_r\}.$
\end{proof}

\begin{Rem} \label{rem83}	   
One can extend the parametrical description used above as follows. We can
assume that the abelian group $L_1/(L_1 \cap N)$ is torsion free, thus
\linebreak $L_1/(L_1 \cap N) \cong \R^{k_0} \times \Z^{l_0},\,
H_{i-1}/H_i \cong \R^{k_i} \times \Z^{l_i}$. Choosing $B_i$ so that the
images in the
quotient group give a basis for $\R^{k_i}$ and a generating set for
$\Z^{l_i}$, one can consider 
$\mathbf y_i^{\mathbf t_i} \text{ for }
\mathbf t_i \in \R^{k_i} \times \Z^{l_i}$. Assuming that $H_r$ is trivial,
it follows that every $x \in L_1 N$ has a unique representation
\,$\mathbf y_0^{\mathbf t_0}\dots \mathbf y_r^{\mathbf t_r}$ which
generalizes the coordinates of the second kind of \cite{Ma}\;\S\,2.
\vspace{2mm}
\end{Rem}
\section{On Gromov's construction} 
\medskip
\begin{varrem}\label{di91}	     
An essential tool for Gromov's proof in \cite{Gr} was the construction of
a ``limit space'' \,(denoted by $Y$ in \cite{Gr}) which can be associated
to any metric
space and is now called an {\it asymptotic cone}. See \cite{KT} for
further discussion and references. For the case of compactly generated
groups $G$ of 
polynomial growth, we want to give an explicit description in terms of
the algebraic hull (see also \cite{Br}\;Th.1.9 and Remark\;\ref{rem93}\,(b)
below).

We use similar notations as in \cite{Lo1}\;p.\,113f. Let $V$ be a
fixed relatively compact symmetric $e$-neighbourhood generating $G$ and
consider \,$\tau = \tau_V$ as in Definition\;\ref{def61} \,(in \cite{Lo1}
this is denoted as
$\lVert x\rVert$). \vspace{.5mm}We consider a fixed non-trivial
ultrafilter $\Cal U$ on~$\N$\,. Put\vspace{-1.5mm}
$\Cal X = \bigl\{(x_n)^{\infty}_{n = 1} \subseteq G\,:\
\dfrac{\tau (x_n)}n \text{ \,is bounded\,}\bigr\}$,
\,$\Cal M_{\Cal U}=\bigl\{(x_n)^{\infty}_{n = 1} \subseteq G\,:\
\lim\limits_{\Cal U}\dfrac{\tau (x_n)}n=0\,\bigr\}$.
Then $\Cal X$ is a group and
$\Cal X_{\Cal U}=\Cal X/\Cal M_{\Cal U}$ (left cosets) is called
the asymptotic cone. It consists of the 
equivalence classes $(x_n)\sptilde$ in $\Cal X$\vspace{.8mm}, where
\,$(x_n) \sim (y_n)
\text{ \,if \,}\lim\limits_{\Cal U} \dfrac{\tau({y_n^{-1}}x_n)}n = 0$\,.
For $d((x_n)\sptilde,(y_n)\sptilde)=
\lim\limits_{\Cal U} \dfrac{\tau({y_n^{-1}}x_n)}n\,,\ \Cal X_{\Cal U}$
becomes a complete
metric space, $\Cal X$ acts on $\Cal X_{\Cal U}$ by left multiplication and
$\Cal Y_{\Cal U}=\Cal X/\Cal N_{\Cal U}$ denotes the resulting group of
transformations on $\Cal X_{\Cal U}$ with
the topology of uniform convergence on the balls of $\Cal X_{\Cal U}$\,,
i.e. $\Cal N_{\Cal U}=\bigl\{(x_n)\in~\Cal X:\
(x_ny_n)\sim(y_n)\text{ for all }(y_n)\in\Cal X\,\bigr\}$.
The group $\Cal Y_{\Cal U}$ consists of isometries of $\Cal X_{\Cal U}$
\linebreak
(i.e., $\Cal Y_{\Cal U} \subseteq \operatorname{Isom} (\Cal X_{\Cal U})$\,)
and acts transitively.

For a connected, simply connected nilpotent Lie group $\widetilde N$
with Lie algebra $\tilde{\fr n}$\,, we consider a decomposition 
\,$\tilde{\fr n} = \bigoplus \limits^r_{j = 1} \fr w_j$ 
as in the proof of Proposition\;\ref{pro64}, \vspace{-2mm}i.e.,
$\bigoplus\limits_{i = 1}^j \fr w_i$ is the Lie algebra of
\,$C_{j-1}(\widetilde N)\ \:(j= 1,\dots,r)$. We introduce a new Lie product
on the vector space of \,$\tilde{\fr n}$\,, taking \,$[X,Y]^{\infty}$ as the
projection of
\,$[X, Y]$ to $\fr w_{i + j}$ for $X \in \fr w_i\,,\;Y \in
\fr w_j\,,\ i, j \geq 1$ and extending bilinearly. The resulting
Lie algebra $\tilde{\fr n}^{\infty}$ is called the {\it graded Lie algebra}
associated to $\tilde{\fr n}$\,, the corresponding simply connected Lie group
is denoted by $\widetilde N^{\infty}$ \,(the graded Lie group associated to
$\widetilde N$\,).
$\tilde{\fr n}^{\infty}$ (resp.~$\widetilde N^{\infty}$) has the same
descending central
series as $\tilde{\fr n}$ (resp. $\widetilde N$). If
\,$\tilde{\fr n} = \bigoplus\limits_{j = 1}^r \fr w_j'$ is another
decomposition as above, 
$u\!: \fr w_1 \to \fr w_1'$ \,any linear isomorphism,
it is easy to see that this induces a unique Lie algebra isomorphism between
the associated graded Lie
algebras \,$\tilde{\fr n}^{\infty}$ and $\tilde{\fr n}^{\infty\,\prime}$.
In this
sense $\tilde{\fr n}^{\infty}$ and $\widetilde N^{\infty}$ are unique.
If $K$ is a
group acting on $\widetilde N$ by automorphisms so that the corresponding
representation on $\tilde{\fr n}$ is semisimple \,(e.g., $K$ a compact group
with a continuous action), one can choose $\fr w_j$ $K$-invariant.
Then $K$ induces automorphisms on $\tilde{\fr n}^{\infty}$ and
$\widetilde N^{\infty}$,
in particular $\widetilde N^{\infty} \rtimes K$ is well defined.
\end{varrem}
\begin{Thm}  \label{th92}     
Let $G$ be a compactly generated l.c.\;group of polynomial growth, $C$~its
maximal compact normal subgroup, $\widetilde G = \widetilde N \rtimes K$
denotes the algebraic
hull of $G/C$. Let $\Cal U$ be any non-trivial ultrafilter on $\N$\,.
Then the asymptotic cone $\Cal X_{\Cal U}$ of $G$ is homeomorphic to
$\widetilde N$. The group of isometries $\Cal Y_{\Cal U}$ is isomorphic to
$\widetilde N^{\infty} \rtimes K$.
\end{Thm}
\begin{proof}
It follows immediately from the properties of $\tau$ and $\Cal X_{\Cal U}$
that the asymptotic cones of $G$ and $G/C$ can be identified. Similarly
(Lemma\;\ref{le63}),
when $G$ is embedded as a closed co-compact subgroup into another group.
Thus (Theorem\;\ref{th2}) it will be enough to consider the case
\,$G = \widetilde G = \widetilde N \rtimes K$\,.
First, we assume that \,$G = \widetilde N$\,.
\\[.3mm plus .3mm]
For \,$X^{\infty} = \sum\limits_{j = 1}^r W_j \in \tilde{\fr n}
= \bigoplus\limits_{j = 1}^r {\fr w}_j$ \,put
\,$y_n = \exp (\sum\limits_{j = 1}^r n^j\,W_j) \quad (n = 1, 2,\dots)$.
We claim that these sequences $(y_n)$ give representatives for all the 
equivalence classes of $\Cal X_{\Cal U}$ and $\Psi(X^{\infty})=(y_n)\sptilde$
defines a homeomorphism $\Psi$ between $\tilde{\fr n}$ and $\Cal X_{\Cal U}$\,.
\\[1.5mm plus .3mm]
Consider the function $\varphi$ on
$\tilde{\fr n}$ as in the proof of Proposition\;\ref{pro64}. Recall
that $\varphi (X_n)$ has equivalent growth to 
\,$\tau (\exp X_n) \text{ for } X_n \to \infty$. Thus for
\,$x_n = \exp X_n$\,,
we have $(x_n) \in \Cal X \text{ iff \,} \dfrac{\varphi (X_n)}n$ \,is
bounded.\vspace{.5mm} In particular, $(y_n)$, as defined above,
belongs to $\Cal X$, i.e.,
$\Psi\!: \tilde{\fr n} \to \Cal X_{\Cal U}$\,. For
$X_0, Y_0 \in \tilde{\fr n}$
put \,$Z_0 = \log(\exp X_0 \exp Y_0)$\,. In the proof of \cite{Gu}\;L.\,II.1,
it was shown that after suitable scaling of the norms on ${\fr w}_j$\,, one
has
\,$\varphi (Z_0) \leq \varphi (X_0) + \varphi (Y_0)$ \,for
\,$\varphi (X_0) + \varphi (Y_0) \geq 2$. By similar arguments, one can show
that 
\;$\varphi (Z_0) \leq
\varphi (X_0 + Y_0)^{\frac1r}
\bigl(\varphi (X_0) + \varphi (Y_0)\bigr)^{1 - \frac 1r}$
\,for \,$\varphi (X_0) + \varphi (Y_0) \geq 2$
\ (using the recursion formulas for the Campbell-Hausdorff series,
\cite{Va}\;(2.15.15), one gets that, in the notation of \cite{Gu} 
\,$\lVert Q (\bar x, \bar y)\rVert \leq
k\,\varphi (\bar x + \bar y)\,\varphi (\bar x)^{r - 1} \text{ \,for }
\varphi (\bar x) \geq \max (1, \varphi (\bar y))$\,). Without
scaling, this implies that there exists some $c > 0$ such that\vspace{-1.3mm}
$$\varphi (Z_0) \leq
c \,\varphi (X_0 + Y_0)^{\frac 1r}\,
\bigl(\varphi(X_0) + \varphi(Y_0)\bigr)^{1-\frac1r}\quad
\text{ for \ }\max\bigl(\varphi (X_0), \varphi (Y_0)\bigr) \geq c\,.
\vspace{-.8mm}$$
Hence if $(X_n),(X_n') \subseteq \tilde{\fr n}$ satisfy $\varphi (X_n),
\varphi (X_n') = O (n),\;\varphi (X_n - X_n') = o(n)$, then
$(\exp (X_n)) \sim (\exp (X_n'))$\,. Furthermore, it follows that
$\Psi$ is continuous for the metric of $\Cal X_{\Cal U}$\,. Given any 
$(x_n) \in \Cal X$, put\vspace{-3.5mm} 
\,$X_n = \log (x_n),\; X_n =\sum\limits_{j = 1}^r W_j^{(n)},\;
W_j^{\infty} = \lim\limits_{\Cal U}\, \dfrac{W_j^{(n)}}{n^j}\,,
\;X^{\infty} = \sum\limits_{j = 1}^r  W_j^{\infty}$.
Then $(x_n)\sptilde=\Psi(X^{\infty})$ and $X^{\infty}$ is the only choice.
Thus $\Psi$ is bijective and a homeomorphism.
\\[-.2mm]
In a similar way as in the proof of Proposition\;\ref{pro64}\,(a), one
can show that there exists some $c > 0$ such that 
\;$\varphi ([X_0, Y_0]) \leq
c\,\max\bigl(\,\varphi(X_0)^{\frac 1r}\,\varphi(Y_0)^{1-\frac 1r}\,,\,
\varphi(X_0)^{1-\frac 1r}\,\varphi(Y_0)^{\frac 1r}\bigr)$
\,for all $X_0, Y_0 \in \tilde{\fr n}$\,. This entails 
\,$\varphi\bigl(\,\ad (y) (X_0)\bigr) \leq c' \varphi (X_0)^{\frac 1r}\,
\tau (y)^{1-\frac 1r} \text{ \,for \,} \tau(y) \geq \varphi (X_0) \geq 1$
\,and then
\;$\tau (y\,x\,y^{-1}) \leq c''\,\tau(x)^{\frac 1r}\,\tau (y)^{1-\frac 1r}$
\,for \,$\tau(y) \geq  \tau (x), \ x, y \in \widetilde N$\,. It 
follows that if \,$(x_n), (y_n) \in X,\
\lim\limits_{\Cal U}\dfrac{\tau (x_n)}n = 0$\,, then
\,$(x_n y_n)\sim (y_n)$\,.\vspace{.3mm}
Thus for $G = \widetilde N$\,, we get that $\Cal M_{\Cal U}$ is
normal in $\Cal X$\,, i.e. $\Cal M_{\Cal U}=\Cal N_{\Cal U}$ and
$\Cal X_{\Cal U}$ is a locally compact group.
\\[0mm plus .3mm]
It follows immediately from the definition that
\,$s \mapsto \Psi (s X^{\infty}) = \Psi (X^{\infty})^s\,\ (s\in\R)$ defines a
one-parameter group in $\Cal X_{\Cal U}$ for
each $X^{\infty} \in \tilde{\fr n}\,.\ \Lie(\Cal X_{\Cal U})$ denoting
the family of all one-parameter groups, we get thus a mapping
$\Psi_0\!:\tilde{\fr n} \to \Lie(\Cal X_{\Cal U})$ which
is bijective. Let $Z_0(s,t)$ be defined
for $s X_0, t Y_0$ as above $(s, t \in \R)$. The Campbell-Hausdorff series
\,(\cite{Va}\;Th.\,2.15.4) and \cite{Va}\;Th.\,3.6.1 express $Z_0(s,t)$
as a polynomial in $s,t$\,. Applying this to the sequences defining
$\Psi (s X^{\infty}), \Psi (t Y^{\infty})$ and computing
 \,$\Psi^{-1}(\Psi (s X^{\infty})\,\Psi (t Y^{\infty}))$ as above, it follows
that
\,$\Psi (s X^{\infty})\,\Psi(t Y^{\infty}) = \Psi\bigl(Z^{\infty}(s,t)\bigr)$
where $Z^{\infty}(s,t)$ is a polynomial in $s, t$ \,(in
particular, when $\widetilde N$ is available, it gives a more direct argument
that
$X_{\Cal U}$ is a Lie group). The linear term of this polynomial is 
\,$s X^{\infty} + t Y^{\infty}$ and consequently\vspace{.3mm} 
\,$\Psi(X^{\infty} + Y^{\infty}) = \lim\limits_{n\to\infty}
\bigl(\Psi(\frac 1n X^{\infty})\,\Psi(\frac 1n Y^{\infty})\bigr)^n$.
One-parameter groups provide a standard tool to define the Lie algebra of a
Lie group (and more generally for connected l.c.\,groups, see also
\cite{Ta}\;sec.\,1.3). By \cite{Va}\;Cor.\,2.12.5, $\Psi_0$ is linear.
Hence the Lie algebra 
$\Lie(\Cal X_{\Cal U})$ is linearly isomorphic to $\tilde{\fr n}$\,.
\\[-.3mm plus .3mm]
For $X^{\infty}\in \fr w_i\,,\;Y^{\infty}\in \fr w_j$\,, we compute the
coefficient of \,$s\,t$ in $Z^{\infty}(s, t)$, the limit procedure contracts
it to the $i + j$\,-component of $[X^{\infty}, Y^{\infty}]$\,, i.e., to
$[X^{\infty}, Y^{\infty}]^{\infty}$. It follows that
$\Lie(\Cal X_{\Cal U})$
is isomorphic (including the Lie bracket) to $\tilde{\fr n}^{\infty}$,
\;$\Psi$ realizes the exponential mapping. The group
structure of \,$\Cal Y_{\Cal U} = \Cal X_{\Cal U}$ is that of
$\widetilde N^{\infty}$ \,(i.e. $\Psi\circ\log$ is a group isomorphism).
\\[1mm plus .6mm]
In the general case, $G = \widetilde N \rtimes K$ \,($K$ acting faithfully),
the asymptotic cone $\Cal X_{\Cal U}$ of $G$ is determined by $\widetilde N$,
hence the description above applies. For $(k_n) \subseteq K$ put
$k = \lim\limits_{\Cal U} k_n$\,. Then for $(x_n) \subseteq \widetilde N$
with $(x_n) \in \Cal X$\,,
it follows from the definitions that $(x_n k_n) \sim (x_n)$ \,and then
\,(transferring to $\tilde{\fr n}$ and working with $\varphi$ as before) that
$(k_n x_n) \sim  (kx_n) \sim (kx_nk^{-1})$. Hence
$\iota(K)$ (see (a) below) is a subgroup of $\Cal Y_{\Cal U}$ isomorphic
to $K$.  Decomposing an arbitrary sequence from $\Cal X$ leads to
\,$\Cal Y_{\Cal U} \cong \widetilde N^{\infty} \rtimes K\,,\
\,\Cal M_{\Cal U}/\Cal N_{\Cal U}\cong K$\,.\vspace{0mm plus1mm}
\end{proof}
\begin{Rems} \label{rem93}	   
(a) \;Consider the homomorphism \,$\iota\!:G \to \Cal Y_{\Cal U}$ which
associates
(classes of) constant sequences to the elements of $G$
\,(i.e., $\iota(x)=(x)\spdot,\ \spdot$ refers to the equivalence class
\!$\mod\Cal N_{\Cal U}$\,). Then Theorem\;\ref{th65}\,(a)
expresses that \,$\ker\,\iota=N$, where $N/C = \nil (G/C)$,
the non-connected nil-radical ($G$~as in Theorem\;\ref{th92}\,).
\vspace{1mm plus1mm}
\item[(b)] The exponential function defines an analytic diffeomorphism between
$\tilde{\fr n}$ and $\widetilde N$ (resp. $\tilde{\fr n}^{\infty}$ and
$\widetilde N^{\infty}$). Since $\tilde{\fr n}$ and $\tilde{\fr n}^{\infty}$
are isomorphic as vector spaces, all these objects may be identified as
manifolds. Differences between $\widetilde N$ and $\widetilde N^{\infty}$
are possible with respect to
the group multiplication (but the one-parameter groups coincide as well).
If the $\fr w_j$ are chosen $K$-invariant , the actions of $K$ will agree
too. Putting $\delta_t(W)=t^jW$ for $W\in\fr w_j\,,\;t>0$ defines
a family $(\delta_t)$ of isomorphisms of $\tilde{\fr n}^{\infty}$ called
dilations. Hence these define also isomorphisms of  $\widetilde N^{\infty}$
and diffeomorphisms of $\tilde{\fr n}$ and $\widetilde N$\,.
Then the homeomorphism $\Psi$
considered in the proof above for $G=\widetilde N$ can be written as
$X^\infty\in\tilde{\fr n}\;\mapsto
\;(\exp(\delta_n(X^\infty))\sptilde\in\Cal X_{\Cal U}$
and $\Psi^{-1}$
is given by $(x_n)\sptilde\in\Cal X_{\Cal U}\;\mapsto\;
\lim\limits_{\Cal U}\delta_{\frac1n}(\log (x_n))\in\tilde{\fr n}$\,.
Similarly for a general compactly generated l.c.\;group of polynomial growth,
with $G,C,\widetilde N$ as in Theorem\;\ref{th92} one has a mapping
(projection)
$\pi_{\widetilde N}:G\to\widetilde N$
(whose image contains a closed co-compact subgroup).
Then $\Psi^{-1}$ is given
by  $(x_n)\sptilde\mapsto
\lim\limits_{\Cal U}\delta_{\frac1n}(\log(\pi_{\widetilde N} (x_n)))$\,.
But note that if $K$ is non-trivial,
this does no longer give a group homomorphism $\Cal X\to\widetilde N^{\infty}$
(\,$\Cal N_{\Cal U}$ is the maximal normal subgroup of $\Cal X$ contained
in $\Cal M_{\Cal U}$\,).
If the $\fr w_j$ are
$K$-invariant, the $\delta_t'=\exp\circ\delta_t\circ\log$ extend to
isomorphisms
of $\widetilde N^{\infty}\rtimes K$\,, putting $\delta_t'(xk)=\delta_t'(x)k$
for $x\in \widetilde N^{\infty},\,k\in K$\,. Then the isomorphism
$\Cal Y_{\Cal U} \to \widetilde N^{\infty} \rtimes K$ is given by \linebreak
$(x_n)\spdot\mapsto\lim\limits_{\Cal U}\delta_{\!\frac1n}'(x_n)\,,\
\Cal N_{\Cal U}=
\{(x_n)\in\Cal X:\lim\limits_{\Cal U}\delta_{\!\frac1n}'(x_n)=e\,\}$ \,and
\,$\Cal M_{\Cal U}=\linebreak
\{(x_n)\in\Cal X:\lim\limits_{\Cal U}\delta_{\!\frac1n}'(x_n)\in K\,\}$\,.
\vspace{1.5mm plus1mm}
\item[(c)] The results on the asymptotic behavior of powers $B^n$ have been
much refined in \cite{Br}. Using the algebraic hull, one can bypass the case
of the simply connected solvable Lie groups treated in sec.\,5 of \cite{Br}
and one can move directly to subsets of semidirect products
$\widetilde N\rtimes K$\,. Also, Cor.\,1.5 of \cite{Br} extends to
arbitrary  compactly generated l.c.\;groups of polynomial growth having no
non-trivial compact normal subgroup.
\\[.3mm]
If $\widetilde N$ is a connected, simply connected nilpotent Lie group,
$\delta_t$ as in (b) and $B$ a non-empty compact subset of $\widetilde N$\,,
one can show, using the methods of \cite{Br}\;sec.\,6, that
$(\delta_{\frac1n}(B^n))$ converges in the standard topology for the
space of non-empty compact subsets of $\widetilde N$\,. In the (much easier)
abelian case this was shown in \cite{EG}\;L.\,5.1. For a non-empty compact
subset $B$ in $\widetilde N\rtimes K$ ($K$ compact), one has also
convergence of $\bigl(\delta_{\frac1n}(\pi_{\widetilde N}(B^n))\bigr)$ \,and
if the image of $B$ in $K$ (considered as a quotient of $\widetilde G$)
generates a dense subgroup, then the limit set is $K$-invariant.
It is easy to give examples (e.g. singletons) where $(\delta_{\frac1n}(B^n))$
does not converge, see also \cite{EG}\;sec.\,6.
If the image of $B$ in $K$ generates a dense subgroup, then convergence holds
iff the
closed normal subgroup generated by the image of $B^{-1}B$ in $K$ is also
dense. Otherwise, there is a circulating behaviour in the $K$-\,component.
\\[1.5mm plus.3mm]
If $G$ is a compactly generated l.c.\;groups of polynomial growth having no
non-trivial compact normal subgroup, $B$ a relatively compact, measurable
$e$-neighbourhood
generating $G$\,, let $\widetilde B$ be the limit of
$(\delta_{\frac1n}({\overline B}^{\;n}))$  in the algebraic hull
$\widetilde G=\widetilde N\rtimes K$. If the Haar measure
$\lambda_{\widetilde G}$ of $\widetilde G$ is chosen so that the corresponding
invariant measure on the quotient space satisfies
$\lambda_{\widetilde G/G}(\widetilde G/G)=1$
one gets similarly as in \cite{Br}\;Cor.\,1.5\,:\quad
$\lim\,\dfrac{\lambda_G(B^n)}{n^d}=\lambda_{\widetilde G}(\widetilde B)$\,.
\vspace{.5mm} Similarly, when $B$ is not measurable with upper and lower
measures
(compare \cite{EG}\;Th.\,7.9). For a general compactly generated l.c.\;group
of polynomial growth $\widetilde B$ can be obtained by considering
the image of $B$ in $G/C$ ($C$~the maximal compact normal subgroup).
If $B^{-1}=B$ then $\widetilde B$ can be
described in terms of a Carnot--Caratheodory metric. Explicit examples
of the limit sets have been given in \cite{Br}\;sec.\,9.
It appears likely that the limit relation is also true when
$B$ is any relatively compact subset of $G$ such that $B$ (or some power)
has non-empty interior and $B^{-1}B$ generates $G$\,.
\vfill
\item[(d)] Theorem\;\ref{th95} shows that for $G$ a compactly generated
l.c.\;group of polynomial growth the spaces $\Cal X_{\Cal U}$ are isomorphic
for each non-trivial ultrafilter $\Cal U$ \,(the same for $\Cal Y_{\Cal U}$).
Hence, additional assumptions
on $\Cal U$ that are sometimes made (e.g. in \cite{Lo1}) to prove the
relevant properties finally turn out to be unnecessary. Also the set
$\Cal X_e=\{(\exp(\delta_n(X^\infty)):\;X^\infty\in\tilde{\fr n}\}$ used
in the proof of Theorem\;\ref{th95} does not depend on $\Cal U$\,. It
provides a section, $\Cal X=\Cal X_e\Cal M_{\Cal U}$\,. But it does not
look easy to give an intrinsic definition of such a set (without knowing
$\widetilde N$).
\end{Rems}

\begin{Def} \label{def94}	
A compactly generated l.c.\;group $G$ is said to be of
{\it near polynomial growth}, if there exist $c, d \in \N$ such that
\,$\lambda (V^{n_i}) \leq c\, n_i^d$
holds for an infinite sequence $(n_i) \subseteq \N$ tending to infinity
\,($\lambda$ Haar measure, $V$ a compact $e$-neighbourhood
generating $G$).
\vspace{1mm}\end{Def}

\begin{Thm}  \label{th95}      
Let $G$ be a compactly generated l.c.\;group. If \,$G$ has near polynomial
growth, then $G$ has polynomial growth.\vspace{-2mm}
\end{Thm}\noindent
For discrete groups, this was shown in \cite{VW} (and with another proof
in \cite{Kl} who used the term weak polynomial growth).
For totally disconnected groups (using the language of graph theory) the
result is in \cite{Tr2}.

\begin{proof}
The argument is essentially the same as in \cite{Lo1}, we just indicate
the changes. Given an ultrafilter $\Cal U$ on $\N$\,, we can define the
asymptotic cone $\Cal X_{\Cal U}$\,.
The proof of \cite{Lo1}\;L.\,8 shows that $\Cal X_{\Cal U}$
is locally compact and finite dimensional if there exist $c_0,d_0$
such that the sets
\,$M_k = \{n \in \N\! :
\lambda(V^n)/\lambda(V^{[ n/{2^k}]}) < c_0\,2^{k d_0}\}$ belong to
$\Cal U$ for $k = 1, 2,\dots$ \;(the assumption
\,$\{2^n: n \geq 1\} \in \Cal U$ was used only to facilitate the formulas;
one has to replace $V^{\frac{nr}4}$ by $V^{[\frac{nr}4]}$ in the proof).
Take $c,d$ as in Definition\;\ref{def94}. We will show next that for
$c_0=2,\;d_0=d+1,\ \,\bigcap\limits_{k=1}^{k_0}M_k$ is infinite for all
$k_0\ge1$ 
\,(then there exists an ultrafilter $\Cal U$ containing these sets).
\\[.5mm]
Assume that \,$k_1=\max\Bigl(\bigcap\limits_{k=1}^{k_0}M_k\Bigr)$ is finite
for some
$k_0\ge1$\,.\;Consider\vspace{-1.5mm} \,$k_2=\linebreak\max(k_1,
2^{k_0+1}d_0)$\,.
Then for every $n>k_2$
there exists $n'=\Bigl[\dfrac n{2^k}\Bigr]$ with $1\le k\le k_0$
\vspace{-1.5mm}\,and
\,$\lambda(V^n)\ge\lambda(V^{n'})\,2^{kd_0+1}\ge
\lambda(V^{n'})\,\Bigl(\dfrac n{n'}\Bigr)^{d_0}$. By iteration, it would
follow that\vspace{-1.8mm}
for every $n>k_2$ there exists $n'\le k_2$ with
\,$\lambda(V^n)\ge\dfrac{\lambda(V^{n'})}{(n')^{d_0}}\,n^{d_0}$ and this
contradicts the condition of Definition\;\ref{def94}.
\\
Now choose $\Cal U$ so that $\Cal X_{\Cal U}$
is locally compact and finite dimensional, then it follows similarly as
in \cite{Lo1} that every closed ball in $\Cal X_{\Cal U}$ is compact and then
that $\Cal Y_{\Cal U}$ is an almost connected Lie group. If $G$ is totally
disconnected, the argument of \cite{Lo1} (see below) shows that $G$ has an
open compact normal subgroup. In the general case, one can see as in
\cite{Gu}\;Th.\,I.3 that near polynomial growth is inherited to
quotient groups (and clearly also to open subgroups). Then as in \cite{Lo1},
it follows that $G$ has a compact normal subgroup $K$ such that $G/K$ is
a Lie group.
\\[1.5mm plus .2mm]
For $G$ is discrete, we basically argue like Gromov.
One has to show the existence of a normal
subgroup $H$ such that $G/H$ is infinite and almost abelian \,(then an
inductive argument, based on a splitting lemma similar to \cite{Gr}\;p.\,59
and lemma\;(b) on \cite{Gr}\;p.\,61
works). Consider $\iota$ as in Remark\;\ref{rem93}\,(a).
More generally, if \,$\zeta=(x_n)\subseteq G$ put
$\iota_{\zeta}(x)=(x_n^{-1}xx_n)$\,. If $\iota(G)$ is infinite, then it is
either almost abelian or
the image in the adjoint group of $\Cal Y_{\Cal U}$ is an infinite linear
group and this can be settled using the Tits alternative. It remains to
consider the case
(after passing to a subgroup of finite index) that $\iota(G)$ is trivial and
$G$ is not almost abelian \,(so $G$ is not an $FC$-group).
Then we want to show (slightly more directly) that there exists $\zeta$ such
that $\iota_{\zeta}(G)\subseteq \Cal X$ and the image in $\Cal Y_{\Cal U}$
becomes
arbitrarily large (then lemma (a) on \cite{Gr}\;p.\,61 applies). This can be
done similarly as in the totally disconnected case (\cite{Lo1}\;p.\,116)\,:
\;For $t>0$ put $m(t)=\linebreak
\min\{m\in\N\!:\ \exists\, x\in V,\,y\in G\!:\;\tau(y)=m\,,\;
\tau(y^{-1}xy)\ge t\,\}$\,. Since $G$ is not an $FC$-group \,(i.e.,
$\{y^{-1}xy:y\in G\}$ is infinite for some $x\in V$), $m(t)$ is finite for
all~$t$\,. We have \,$\tau(y^{-1}zy)<t+2$ for all $z\in V,\ y\in G$ with
$\tau(y)\le m(t)$\,. Triviality of $\iota(G)$ implies $m(t)/t\to\infty$
for $t\to\infty$\,. Since $\Cal Y_{\Cal U}$ is a Lie group, there exist
$l\in\N\,,\,\varepsilon>0$ such that
\,$\{\eta\in \Cal Y_{\Cal U}\!:\,\lVert\eta\rVert_l<\varepsilon\,\}$
contains no non-trivial subgroup of~$\Cal Y_{\Cal U}$
\,($\lVert(x_n)\spdot\rVert_l=
\lim\limits_{n\in\Cal U}\dfrac{\lVert x_n\rVert_{ln}}n$\,, this gives a basis
for the neighbourhoods in $\Cal Y_{\Cal U}$\,; on $G\ \lVert\ \rVert_n$ is
taken from Definition\;\ref{def61}). Now choose $p\in\N$\,.
For \,$m(n\varepsilon/p)<ln$ put $y_n=e$\,, otherwise
\,(splitting an element giving the minimum in the definition of
$m(n\varepsilon/p)$\,) \,there exists $y_n\in G$
such that \,$\tau(y_n)=m(n\varepsilon/p)-ln$ \,and
\,$n\varepsilon/p\le\sup\{\lVert y_n^{-1}xy_n\rVert_{ln}\!:\,x\in V\}<
n\varepsilon/p+2$\,.
Then $\zeta=(y_n)$ satisfies \,$\iota_{\zeta}(V)\subseteq \Cal X$\,, hence
$\iota_{\zeta}(G)\subseteq \Cal X$ \,and
$\lVert\iota_{\zeta}(x)\spdot\rVert_l=\varepsilon/p$
for some $x\in V$ \,(observe that $V$ is finite).
Thus the image of $\iota_{\zeta}(G)$ in $\Cal Y_{\Cal U}$
contains at least $p$ elements.
\\[.2mm plus .1mm]
It remains to consider the case where $G$ is a Lie group.
One can verify the conditions of \cite{Lo1}\;Th.\,1\,: \ $G/G^0$ has
polynomial growth by the argument above. If \,$x,H,U$ are as in
\cite{Lo1}\;L.\,3, then near polynomial growth implies
\,$\lambda_H(x^{n_i}Ux^{-n_i})=O(n_i^k)$ for some infinite sequence $(n_i)$
and then the proof given on \cite{Lo1}\;p.\,112 for \,(a)$\Rightarrow$(b)
works to show type $R_G$\,.
\end{proof}

For discrete, finitely generated groups (and $K$-approximate groups) this has
has been much improved in the work of Breuillard, Green, Tao (see
\cite{Ta}\;Th.\,1.10.1, Th.\,1.10.10 and Exerc.\,1.10.2).

\end{document}